\documentclass[12pt, twoside]{article}
\usepackage{amsmath,amsthm,amssymb}
\usepackage{times}
\usepackage{enumerate}
\usepackage{bm}
\usepackage[all]{xy}
\usepackage{mathrsfs}
\usepackage{amscd}
\usepackage{color}
\def\titlerunning#1{\gdef\titrun{#1}}
\makeatletter
\def\author#1{\gdef\autrun{\def\and{\unskip, }#1}\gdef\@author{#1}}
\def\address#1{{\def\and{\\\hspace*{18pt}}\renewcommand{\thefootnote}{}%
\footnote {#1}}%
\markboth{\autrun}{\titrun}}
\makeatother
\def\email#1{e-mail: #1}

\def\keywords#1{\par\medskip
\noindent\textbf{Keywords.} #1}

\newtheorem{theorem}{Theorem}[section]
\newtheorem{corollary}[theorem]{Corollary}
\newtheorem{lemma}[theorem]{Lemma}
\newtheorem{proposition}[theorem]{Proposition}
\theoremstyle{definition}
\newtheorem{definition}[theorem]{Definition}
\newtheorem{remark}[theorem]{Remark}

\numberwithin{equation}{section}

\xyoption{all}

\def \C {\mathbb{C}}

\def \a {\alpha }
\def \b {\beta}

\def \De {\Delta}
\def \la {\lambda}
\def \La {\Lambda}
\def\w {\omega}
\def\Om{\Omega}

\def\pa{\partial}
\def\na {\nabla}
\def\Ga{\Gamma}

\begin{document}
\baselineskip=17pt

\titlerunning{$L^{2}$ vanishing theorem on some K\"{a}hler manifolds}
\title{$L^{2}$ vanishing theorem on some K\"{a}hler  manifolds}

\author{Teng Huang}

\date{}

\maketitle

\address{T. Huang: School of Mathematical Sciences, University of Science and Technology of China; Key Laboratory of Wu Wen-Tsun Mathematics, Chinese Academy of Sciences, Hefei, Anhui, 230026, P. R. China; \email{htmath@ustc.edu.cn;htustc@gmail.com}}

\begin{abstract}
Let $E$ be a Hermitian vector bundle over a complete K\"{a}hler manifold $(X,\w)$, $\dim_{\C}X=n$, with a $d$(bounded) K\"{a}hler form $\w$, $d_{A}$ be a Hermitian connection on $E$. The goal of this article is to study the $L^{2}$-Hodge theory on the vector bundle $E$. We extend the results of Gromov's \cite{Gro} to the Hermitian vector bundle. At last, as an application, we prove a gap result for Yang-Mills connection on the bundle $E$ over $X$.
\end{abstract}
\keywords{Hodge theory; $d$(bounded) form; Hermitian vector bundle; gauge theory}
\section{Introduction}
Let $(X,\w)$ be a complete K\"{a}hler manifold, $\dim_{\C}X=n$. A basic question, pertaining both to the function theory and topology of $X$, is: when are there non-trivial harmonic forms on $X$, in the various bidegree $(p,q)$ determined by the complex structure? When $X$ is not compact, a growth condition on the harmonic forms at infinity must be imposed, in order that the answer to this question  be useful. A natural growth condition is square-integrability, if $\Om^{p,q}_{(2)}(X)$ denotes the $L^{2}$-forms of type $(p,q)$ on $X$ and $\mathcal{H}^{p,q}_{(2)}(X)$ the harmonic forms in $\Om^{p,q}_{(2)}(X)$. One version of this basic question is: what is the structure of $\mathcal{H}^{p,q}_{(2)}(X)$?

The Hodge theorem for compact manifolds states that every de Rahm cohomology class of a compact manifold $X$ is represented by a unique harmonic form. That is, the space of solutions to the differential equation $(d+d^{\ast})\a=0$ on $L^{2}$-forms over $X$ is a space that depends on the metric on $X$. This space is canonically isomorphic to the purely topological de Rahm cohomology space of $X$. The study of $\mathcal{H}^{p,q}_{(2)}(X)$,a question of so-called $L^{2}$-cohomology of $X$, is rooted in the attempt to extend  Hodge theory to non-compact manifolds. No such result holds in general for complete non-compact manifolds, but there are numerous partial results about the $L^{2}$-cohomology of non-compact manifold. The study of the $L^{2}$-harmonic forms on a complete Riemannian manifold is a very  fascinating and important subject. It also has numerous applications in the field of Mathematical Physics (see for example \cite{Hit,Ete1,Ete2,EJ,SS}).

A differential form $\a$ on a Riemannian manifold $(X,\rm{g})$ is bounded with respect to the Riemannian metric $\rm{g}$ if the $L^{\infty}$-norm of $\a$ is finite,
$$\|\a\|_{L^{\infty}(X)}:=\sup_{x\in X}\|\a(x)\|_{\rm{g}}<\infty.$$
We say that $\a$ is $d$(bounded) if $\a$ is the exterior differential of a bounded form $\b$, i.e., $\a=d\b$ and $\|\b\|_{L^{\infty}(X)}<\infty$.

In \cite{Gro}, Gromov's theorem \cite{Gro} states that if the K\"{a}hler form $\w$ on a complete K\"{a}hler manifold satisfies $\w=d\theta$, where $\theta$ is a bounded $1$-form,  then the only  non-trivial $L^{2}$-harmonic form lies in the middle dimension. Jost-Zuo \cite{JZ} and Cao-Xavier \cite{CX} extended Gromov's theorem to the case of sublinear growing $1$-form $\theta$. Other results on $L^{2}$ cohomology can be found in \cite{And,Dod,McN1,McN2}.

In this article, we consider the Hodge theory on a Hermitian vector bundle $E$ over a complete, K\"{a}hler manifold $X$, $\dim_{\C}X=n$, with a K\"{a}hler form $\w$. Define a smooth K\"{a}hler metric, $g(\cdot,\cdot)=\w(\cdot,J\cdot)$ on $X$, where $J$ is the complex structure on $X$. Let $d_{A}$ be a Hermitian connection on $E$. The formal adjoint operator of $d_{A}$ acting on $\Om^{k}(X,E):=\Om^{k}(X)\otimes E$  is $d^{\ast}_{A}=-\ast d_{A}\ast$, where $\Om^{k}(X)$ is smooth $k$-forms on $X$ and $\ast$ is the Hodge star operator with respect to the metric $g$. We denote by $\mathcal{H}^{k}_{(2)}(X,E)$ the space of $L^{2}$ harmonic forms in  $\Om^{k}(X)$ with respect to the Laplace-Beltrami operator $\De_{A}:=d_{A}d_{A}^{\ast}+d_{A}^{\ast}d_{A}$, See Definition \ref{D3.1}. 

If $\w$ is $d$(bounded), we can extend the vanishing theorem of Gromov's to Hermtian vector bundle. 
\begin{theorem}\label{T1}
Let $(X,\w)$ be a complete, K\"{a}hler manifold, $\dim_{\C}X=n$, with a $d$(bounded) K\"{a}hler form $\w$, $E$ be a Hermitian vector bundle over $X$ and $d_{A}$ be a Hermitian connection on $E$.  Then a smooth form $\a\in\Om^{p,q}(X,E)\cap\mathcal{H}^{p+q}_{(2)}(X,E)$ is zero unless $p+q=n$.
\end{theorem}
\begin{remark}
We set the bundle $E$ to $TX\otimes\C$, then $d_{A}=d$ and $\bar{\pa}_{A}=\bar{\pa}$.  We denote by $$\mathcal{H}^{k}_{(2)}(X):=\{\a\in\Om^{k}_{(2)}(X):\De\a=(dd^{\ast}+d^{\ast})\a=0 \}$$ 
the space of $L^{2}$-harmonic $p$-forms and  $$\mathcal{H}^{p,q}_{(2)}(X):=\{\a\in\Om^{p,q}_{(2)}(X):\De_{\bar{\pa}}\a=(\bar{\pa}\bar{\pa}^{\ast}+\bar{\pa}^{\ast}\bar{\pa})\a=0 \}$$ 
the space of $L^{2}$-harmonic $(p,q)$-forms. By the  K\"{a}hler identities and $\De=2\De_{\bar{\pa}}$, we have $$\Om^{p,q}(X,E)\cap\mathcal{H}^{p+q}_{(2)}(X,E)\cong\mathcal{H}_{(2)}^{p,q}(X).$$
In this case, we also have
\begin{equation}\label{E15}
\mathcal{H}^{k}_{(2)}(X,E)=\bigoplus_{p+q=k}\mathcal{H}^{p,q}_{(2)}(X,E).
\end{equation} 
Therefore, $\mathcal{H}^{k}_{(2)}(X,E)=\{0\}$. But unfortunately, for other cases,  the Equation (\ref{E15}) is alway incorrect, since $\De_{A}=2\De_{\bar{\pa}_{A}}$ is alway incorrect.
\end{remark} 
We denote by $\mathcal{A}^{1,1}_{E}$ the space of all integrable connections $d_{A}$, i.e., $F^{2,0}_{A}=F_{A}^{0,2}=0$. The important observation is that if the Hermitian connection $d_{A}\in\mathcal{A}^{1,1}_{E}$, then the operator $L^{k}$ could commutes with $\De_{A}$ for any $k\in\mathbb{N}^{+}$. Following the idea in \cite{Gro}, we can prove a vanishing theorem on the spaces $\mathcal{H}^{k}_{(2)}(X,E)$. In \cite{Gro}, Gromov also gave a lower bound on the spectra of the Laplace operator $\De:=dd^{\ast}+d^{\ast}d$ on $L^{2}$-forms $\Om^{p}(X)$ for $p\neq n$ to sharpen the Lefschetz vanishing theorem. We also obtain a lower bound on the spectra of the Laplace operator $\De_{A}$ on $L^{2}$-forms $\Om^{p}(X)\otimes E$ for $p\neq n$ to sharpen the vanishing theorem \ref{T3}.   Next, we will consider the $(0,2)$-part of the curvature, $F_{A}^{0,2}$, and suppose it to be non-zero. In this case $\De_{A}$ does not commute with $L$. We also suppose that the $(0,2)$-part of the curvature obeys $\|F^{0,2}_{A}\|_{L^{n}(X)}\leq\varepsilon$, where $\varepsilon$ is a small enough positive constant and the Riemannian curvature tensor is bounded. We can prove that
\begin{theorem}\label{T3}
Let $(X,\w)$ be a complete, K\"{a}hler manifold, $\dim_{\C}X=n$, with a $d$(bounded) K\"{a}hler form $\w$, i.e., there is a bounded $1$-form $\theta$ such that $\w=d\theta$, $d_{A}$ be a smooth Hermitian connection on a Hermitian vector bundle $E$ over $X$. Suppose that the connection $d_{A}$ and the K\"{a}hler metric $g$ obey one of following sets of conditions: \\
(1) $d_{A}\in\mathcal{A}_{E}^{1,1}$,  \\
(2) $d_{A}\notin\mathcal{A}_{E}^{1,1}$ but the curvature $F_{A}$ of the connection $d_{A}$ obeys

 (a)  $\|F^{0,2}_{A}\|_{L^{n}(X)}\leq\varepsilon$ or 
 
 (b) $\|F^{0,2}_{A}\|_{L^{\infty}(X)}\leq\varepsilon$,\\
where $\varepsilon=\varepsilon(X,n,\theta)\in(0,1]$ is positive constant, the Riemannian curvature tensor of $g$ is bounded.\\
If $\a\in\Om^{k}_{(2)}(X,E)$ such that $\De_{A}\a\in L^{2}$, ($k\neq n)$, then we have the inequality
\begin{equation}\label{E08}
c_{n,k}\|\theta\|^{-2}_{L^{\infty}(X)}\|\a\|^{2}_{L^{2}(X)}\leq\langle\a,\De_{A}\a\rangle_{L^{2}(X)},
\end{equation}
where $c_{n,k}>0$ is a constant which depends only on $n,k$. In particular, we have 
$$\mathcal{H}^{k}_{(2)}(X,E)=0,\ \forall\ k\neq n.$$
\end{theorem}
If there exists a differential form $\theta$ and a number $c>0$ such that the K\"{a}hler form $$\w=d\theta,\ \|\theta(x)\|_{L^{\infty}(X)}\leq c(1+\rho(x_{0},x)),$$
where $\rho(x_{0},x)$ stands for the Riemannian distance between $x$ and a base point $x_{0}$, then we call the K\"{a}hler form $\w$  $d$(sublinear) \cite{CX,JZ}.

Since $\De_{A}$ commutes with $L$ for any $d_{A}\in\mathcal{A}^{1,1}_{E}$, following the idea in \cite{CX}, we can prove that 
$$L\a=\a\wedge\w=0$$ 
for any $L^{2}$-harmonic form $\a$.  We then have
\begin{theorem}\label{T5}
Let $(X,\w)$ be a complete, K\"{a}hler manifold, $\dim_{\C}X=n$, with a $d$(sublinear) K\"{a}hler form $\w$, $d_{A}$ be a smooth Hermitian connection on $E$. If $d_{A}\in\mathcal{A}_{E}^{1,1}$, then the spaces of $L^{2}$-harmonic forms $\mathcal{H}^{k}_{(2)}(X,E)$  for $k\neq n$ vanish.
\end{theorem}
Let $(X,\w)$ be a complete K\"{a}hler manifold, $\dim_{\C}X=n$, $E$ be a holomorphic Hermitian vector bundle on $X$.  We denote by $D_{E}=\pa_{E}+\bar{\pa}_{E}$ its Chern connection, i.e., $\bar{\pa}_{E}=\bar{\pa}$, by $D^{\ast}_{E}$ the formal adjoint of $D_{E}$ and by $\pa_{E}^{\ast}$, $\bar{\pa}_{E}^{\ast}$ the components of $D_{E}^{\ast}$ of type $(-1,0)$ and $(0,-1)$. Let $\Theta(E)=\pa_{E}\bar{\pa}_{E}+\bar{\pa}_{E}\pa_{E}$ be the curvature operator on $E$. It is clear that $\bar{\pa}^{2}=0$. Therefore, for any integer $p=0,1,\ldots,n$, we get a complex
$$\Om^{p,0}(X,E)\xrightarrow{\bar{\pa}}\ldots\xrightarrow{\bar{\pa}}\Om^{p,q}(X,E)\xrightarrow{\bar{\pa}}\Om^{p,q+1}(X,E)\rightarrow\ldots,$$
known as the Dolbeault complex of $(p,\bullet)$-forms with values in $E$. We can define two other operators: $\De_{\bar{\pa}_{E}}:=\bar{\pa}_{E}\bar{\pa}_{E}^{\ast}+\bar{\pa}_{E}^{\ast}\bar{\pa}_{E}$ and $\De_{\pa_{E}}:=\pa_{E}\pa_{E}^{\ast}+\pa_{E}^{\ast}\pa_{E}$. Let us introduce, See \cite[Charp V]{Dem}
$$\mathcal{H}^{p,q}_{(2)}(X,E):=\{\a\in\Om^{p,q}_{(2)}(X,E):\De_{\bar{\pa}_{E}}\a=0\}.$$
There are many vanishing theorems for Hermitian vector bundles over a compact complex manifolds. All these theorems are based on  a priori inequality for $(p, q)$-forms with values in a vector
bundle, known as the Bochner-Kodaira-Nakano inequality. This inequality naturally leads to several positivity notions for the curvature of a vector bundle (\cite{Gri69,Kod53a,Kod53b,Kod54,Nak55,Nak73}). We will prove another vanishing theorem for Hermitian vector bundles over the complete K\"{a}hler manifolds with a $d$(bounded) K\"{a}hler form $\w$, only under the curvature $\Theta(E)$ has a small enough $L^{n}$- or $L^{\infty}$-norm. 
\begin{theorem}\label{T1.4}
Let $(X,\w)$ be a complete, K\"{a}hler manifold, $\dim_{\C}X=n$, with a $d$(bounded) K\"{a}hler form $\w$, i.e., there is a bounded $1$-form $\theta$ such that $\w=d\theta$, $D_{E}$ be the Chern connection on a holomorphic Hermitian vector bundle $E$ over $X$. Suppose that the Riemannian curvature tensor of the K\"{a}hler metric $g$ on $X$ is bounded. Then there is a constant $\varepsilon=\varepsilon(X,n,\theta)\in(0,1]$ with following significance. If the curvature $\Theta(E)$ obeys\\
(1) $\|\Theta(E)\|_{L^{n}(X)}\leq \varepsilon$ or\\
(2) $\|\Theta(E)\|_{L^{\infty}(X)}\leq \varepsilon$\\
then any $\a\in\Om^{p,q}_{(2)}(X,E)$ such that $\De_{\bar{\pa}_{E} }\a\in L^{2}$, which satisfies the inequality
\begin{equation}
c_{n,k}\|\theta\|^{-2}_{L^{\infty}(X)}\|\a\|^{2}_{L^{2}(X)}\leq\langle\De_{\bar{\pa}_{E}}\a,\a\rangle_{L^{2}(X)}.
\end{equation}
where $c_{n,k}>0$ is a constant which depends only on $n,k$.\ In particular,\ we have
$$\mathcal{H}^{p,q}_{(2)}(X,E)=0,\ \forall\ p+q\neq n.$$
\end{theorem}
If the background manifold $X$ is a complete K\"{a}hler surface, we don't know any thing about the space of $\mathcal{H}^{2}_{(2)}(X,E)$. In \cite{Hit}, Hitchin proved that any $L^{2}$-harmonic forms of degree $2$ on a hyperbolic K\"{a}hler surface with a $d$(sublinear) growth K\"{a}hler form $\w$, is anti-self-dual. We will also show that if $\a\in \mathcal{H}^{2}_{(2)}(X,E)$, then $\a$ is anti-self-dual under some certain conditions on $X,d_{A}$.
\begin{theorem}\label{T1.6}
Let $(X,\w)$ be a complete K\"{a}hler surface with a $d$(bounded) K\"{a}hler form $\w$, i.e., there is a bounded $1$-form $\theta$ such that $\w=d\theta$, $d_{A}$ be a Hermitian connection on a  Hermitian vector bundle $E$ over $X$. Suppose that the scalar curvature of the K\"{a}hler metric $g$ on $X$ is non-negative. Then there is constant $\varepsilon=\varepsilon(X,n,\theta)\in(0,1]$ with following significance. If the curvature $F_{A}$ of the connection $d_{A}$ obeys
$$\|F_{A}^{+}\|_{L^{2}(X)}\leq \varepsilon$$
then any  $L^{2}$-harmonic form with respect to  $\De_{A}$ is anti-self-dual.
\end{theorem}
In differential geometry and gauge theory, the study of  Yang-Mills connection $d_{A}$ is very important. In fact, the curvature $F_{A}$ of the Yang-Mills connection is a harmonic $2$-form on the bundle $\mathfrak{g}_{E}\subset End E$ with respect to the operator $\De_{{A\otimes A^{\ast}}}$, where $d_{A\otimes A^{\ast}}$ is the connection on $End E$ induced by the connection $d_{A}$ on $E$. There are very few gap results for Yang-Mills connection over non-compact, complete manifolds, for example \cite{DM,EFM,Ger,Min,She,Xin}. Those results depend on some positive conditions of Riemannian curvature tensors. Following the second item of Theorem \ref{T3} and Theorem \ref{T1.6}, we have a gap result for Yang-Mills connection on a complete K\"{a}hler manifold with a K\"{a}hler $d$(bounded) form $\w$.
\begin{corollary}\label{C1}
Let $(X,\w)$ be a complete, K\"{a}hler manifold, $\dim_{\C}X=n$, with a $d$(bounded) K\"{a}hler form $\w$, i.e., there exists a bounded $1$-form $\theta$ such that $\w=d\theta$, $E$ be a Hermitian vector bundle over $X$. If the curvature $F_{A}$ of the smooth Yang-Mills connection $d_{A}$ is in $L^{2}$, then there is a  constant $\varepsilon=\varepsilon(X,n,\theta)\in(0,1]$ with following significance.\\
(1) Suppose that $n\geq 3$ and the Riemannian curvature of the K\"{a}hler metric $g$ is bounded, then either $$\|F_{A}\|_{L^{n}(X)}\geq\varepsilon$$ or the connection $d_{A}$ is flat.\\
(2) Suppose that $n=2$ and the scalar curvature of the K\"{a}hler metric $g$ is non-negative , then either $$\|F_{A}^{+}\|_{L^{2}(X)}\geq\varepsilon$$ or the connection $d_{A}$ is anti-self-dual.
\end{corollary}
\begin{remark}
There are many complete K\"{a}hler surfaces with a $d$(bounded) K\"{a}hler from $\w$. If we also suppose that there exists a non-trivial $L^{2}$-harmonic form $\a$ on $(X,\w)$ is anti-self-dual, then the form $\a$ can be identified with the curvature of an $U(1)$ instanton $d_{\Ga}$ by writing $F_{\Ga}=\sqrt{-1}\a$. We regard this as a reducible $SU(2)$ instanton $d_{A}=d_{\Ga}\oplus d_{-\Ga}$ on the the splitted bundle $E=L\oplus L^{-1}$ over $(X,\w)$ with finite Yang-Mills energy $$YM(A)=\|F_{A}\|^{2}_{L^{2}(X)}=2\|F_{\Ga}\|^{2}_{L^{2}(X)}=2\|\a\|^{2}_{L^{2}(X)}.$$
Furthermore, if $L$ is a trivial bundle $L\cong X\times\mathbb{C}$, the $c\a$ is also a non-trivial anti-self-dual form on $(X,\w)$ for any $c\in\mathbb{R}\backslash\{0\}$. Hence we can also construct a reducible $SU(2)$ instanton $d_{A_{c}}=d_{\Ga_{c}}\oplus d_{-\Ga_{c}}$ on the splitted bundle $E=L\oplus L^{-1}$ over $(X,\w)$ with Yang-Mills energy $2c^{2}\|\a\|^{2}_{L^{2}(X)}$. This unexpected continuous energy phenomenon occurs, for example,  in the important case of K\"{a}hler surface with a $d$(bounded) K\"{a}hler from $\w$.
\end{remark}

\section{Preliminaries}
\subsection{Hermitian exterior algebra}
Let $X$ be a smooth K\"{a}hler manifold with K\"{a}hler form $\w$ and $E$ be a smooth vector bundle over $X$.  We denote by  $\Om^{k}(X,E)$ the space of $C^{\infty}$  sections of the tensor product vector bundle  $\Om^{k}(X)\otimes E$ obtained from $\Om^{k}(X)$ and $E$, i.e., $\Om^{k}(X,E):=\Gamma(\Om^{k}(X)\otimes E)$. We denote by $\Om^{p,q}(X,E)$ the space of $C^{\infty}$ sections of the bundle $\Om^{p,q}(X)\otimes E$. We have therefore a direct sum decomposition
$$\Om^{k}(X,E)=\bigoplus_{p+q=k}\Om^{p,q}(X,E).$$
For any connection $d_{A}$ on $E$, we have the covariant exterior derivatives 
$$d_{A}:\Om^{k}(X)\otimes E\rightarrow\Om^{k+1}(X)\otimes E.$$
Like the canonical splitting the exterior derivatives $d=\pa+\bar{\pa}$, $d_{A}$ decomposes over $X$ into $$d_{A}=\pa_{A}+\bar{\pa}_{A}.$$
We will need some of the basic apparatus of Hermitian exterior algebra. Denote by $L$ the operator of exterior multiplication by the K\"{a}hler form $\w$:
$$L\a=\w\wedge\a,\ \a\in\Om^{p,q}(X,E),$$
and, as usual, let $\La$ denote its pointwise adjoint, i.e.,
$$\langle\La\a,\b\rangle=\langle\a,L\b\rangle.$$
Then it is well known that $\La=\ast^{-1}\circ L\circ \ast$ \cite{Huy}. We denote by $P^{k}$ the space of primitive $k$-forms on $E$, that is
$$P^{k}=\{\a\in\Om^{k}(X,E):\La\a=0\}.$$
A basic fact is
\begin{lemma}\label{L1}
Let $k=p+q$,\\
(i) if $k>n$, then $P^{k}=0$.\\
(ii) if $k\leq n$, then
$$P^{k}=\{\a\in\Om^{k}(X,E):L^{n-k+1}\a=0\}.$$
Furthermore, if $\a\in P^{p,q}:=P^{k}\cap\Om^{p,q}(X,E)$, then
$$\ast L^{n-k}\a/(n-k)!=(\sqrt{-1})^{p^{2}-q^{2}}(-1)^{pq}\a.$$
(iii) The map $L^{n-k}:P^{k}\rightarrow\Om^{2n-k}(X,E)$ is injective for $k\leq n$.\\
(iv) The map $L^{n-k}:\Om^{k}(X,E)\rightarrow\Om^{2n-k}(X,E)$ is bijective for $k\leq n$.
\end{lemma}
The proof is then purely algebraic and can be found in standard texts on geometry. An elegant approach is through representations of $sl_{2}$, see \cite[ Chap.5, Theorem 3.12]{Wel} or \cite{Dem,Huy}.

\subsection{K\"{a}hler hyperbolic manifold}
In \cite{Gro}, Gromov introduced a new notion of hyperbolicity. More specifically, following Gromov, we say that a differential  form $\a$ on a Riemannian manifold is $d$(bounded) if there exists a bounded form $\b$ such that $\a=d\b$. We say that $\a$ is $\tilde{d}$(bounded) if its lift to the universal covering is $\tilde{d}$(bounded). A closed complex manifold is then called K\"{a}hler hyperbolic if the universal covering space admits a K\"{a}hler metric whose K\"{a}hler form is $\tilde{d}$(bounded). The most important examples are compact complex manifolds which admit a K\"{a}hler metric of negative sectional curvature. Gromov used this notion to show that on the universal covering space of a K\"{a}hler hyperbolic manifold, there are no (non-trivial) $L^{2}$-harmonic forms outside the middle degree. In \cite{Chen}, the author showed that the space of $L^{2}$ harmonic forms in the middle degree is infinite dimensional. A key argument in Gromov's \cite{Gro} proof is the following theorem. 
\begin{theorem}\label{T4}
Let $(X,\w)$ be a complete, K\"{a}hler manifold, $\dim_{\C}X=n$, with a $d$(bounded) K\"{a}hler form $\w$, i.e., there is a bounded $1$-form $\theta$ such that $\w=d\theta$. Then every $L^{2}$-form $\a$ on $X$ of degree $p\neq n$ satisfies the inequality
$$c_{n}\|\theta\|^{-2}_{L^{\infty}(X)}\|\a\|^{2}_{L^{2}(X)}\leq\langle\a,\De\a\rangle_{L^{2}(X)},$$
where $c_{n}>0$ is a constant which depends only on $n$. In particular, 
$$\mathcal{H}^{p}_{(2)}(X)=\{0\},\ \forall p\neq n.$$
\end{theorem}
We recall a special Stokes's Theorem for complete Riemannian Manifolds \cite{Gaf,Gro}.
\begin{theorem}\label{T2.3}
Let $\eta$ be an $L^{1}$-form on $X$ of degree $n-1$, i.e., $$\|\eta\|_{L^{1}(X)}=\int_{X}|\eta|dvol<\infty,$$ such that the differential $d\eta$ is also $L^{1}$. If $X$ is complete, then
$$\int_{X}d\eta=0.$$
\end{theorem}
We recall some definitions on Hermitian vector bundle \cite[Charp V, Section 7]{Dem}. Let $E$ be a Hermitian vector bundle of rank $r$ over a smooth Riemannian manifold $X$, $\dim_{\mathbb{R}}X=n$. We denote respectively by $(\xi_{1},\ldots,\xi_{n})$ and $(e_{1},\ldots,e_{r})$ orthonormal frames on $TX$ and $E$ over an open subset $U\subset X$. The associated inner product of $E$ given by a positive definite Hermitian metric $h_{\la\mu}$ with smooth coefficients on $U$, such that
$$\langle e_{\la}(x),e_{\mu}(x)\rangle=h_{\la\mu}(x),\ \forall x\in\Om.$$
When $E$ is Hermitian, one can define a natural sesquilinear map
$$\Om^{p}(X,E)\times\Om^{q}(X,E)\rightarrow\Om^{p+q}(X,\mathbb{C})$$
$$(\a,\b)\mapsto tr(s\wedge t)$$
combining the wedge product of forms with the Hermitian metric on $E$. If $\a=\sum\sigma_{\la}\otimes e_{\la}$, $\b=\sum\tau_{\mu}\otimes e_{\mu}$, we let
$$tr(\a\wedge\b):=\sum_{1\leq\la,\mu\leq r}\sigma_{\la}\wedge\bar{\tau}_{\mu}\langle e_{\la},e_{\mu}\rangle.$$
A connection $d_{A}$ said to be compatible with the Hermitian structure of $E$, or briefly Hermitian, if for every $\a\in\Om^{p}(X,E)$, $\b\in\Om^{q}(X,E)$ we have
\begin{equation}\label{E05}
dtr(\a\wedge\b)=tr(d_{A}\a\wedge\b)+(-1)^{p}tr(\a\wedge d_{A}\b).
\end{equation}
The inner product $\langle\cdot,\cdot\rangle$ on $\Om^{\ast}(X,E)$ defined as, See \cite[Charp VI, Section 3.1]{Dem}
$$\langle\a,\b\rangle=\ast tr(\a\wedge\b),\ \ \a,\b\in\Om^{p}(X,E).$$
We denote by $Tr$ the sesquilinear map  $Tr:\Om^{p}(X,EndE)\times\Om^{q}(X,End E)\rightarrow\Om^{p+q}(X,\mathbb{C})$ induced by the map $tr:\Om^{p}(X,E)\times\Om^{q}(X,E)\rightarrow\Om^{p+q}(X,\mathbb{C})$.
Then we have a simple corollary of the above theorem. 
\begin{proposition}\label{P2}
Let $(X,\w)$ be a complete, K\"{a}hler manifold, $\dim_{\C}X=n$, with a $d$(bounded) K\"{a}hler form $\w$, i.e., there is a bounded $1$-form $\theta$ such that $\w=d\theta$, $d_{A}$ be a connection on a vector bundle $E$ over $X$. If $s\in\Om^{p}_{(2)}(X,E)$, ($p\leq n-1)$ satisfies $d_{A}s=0$, then
$$\int_{X}tr(s\wedge s)\wedge\w^{n-p}=0.$$
\end{proposition}
\begin{proof}
At first, following the Equation (\ref{E05}), we observe an identity
$$tr(s\wedge s)\wedge\w^{n-p}=d\big{(}tr(s\wedge s)\wedge\w^{n-p-1}\wedge\theta\big{)}.$$
It is easy to see
$$|tr(s\wedge s)\wedge\w^{n-p}|\leq C(n,p)|s|^{2}$$
and
$$|tr(s\wedge s)\wedge\w^{n-p-1}\wedge\theta|\leq C(n,p)|s|^{2}|\cdot\theta|_{L^{\infty}}.$$
It implies that
$|tr(s\wedge s)\wedge\w^{n-p}|$ and $|tr(s\wedge s)\wedge\w^{n-p-1}\wedge\theta|$ are all in $L^{1}$. Thus following the Theorem \ref{T2.3}, we complete the proof of this proposition.
\end{proof}
\subsection{Harmonic forms of degree 2}
We shall generally adhere to the now standard gauge-theory conventions and notation of Donaldson and Kronheimer\cite[Section 2.1]{Donaldson/Kronheimer}. Let $X$ be a manifold with a Riemannian metric $g$. Let $E$ be a vector bundle over $X$ with a compact Lie group as its structure group. We denote by $\mathfrak{g}_{E}$ the bundle of Lie algebras associated to the adjoint representation. Thus, $\mathfrak{g}_{E}$ is a subbundle of $End E=E\otimes E^{\ast}$. A connection $d_{A}$ of $E$ can be given by specifying a covariant derivative $d_{A}:\Om^{p}(X,E)\rightarrow\Om^{p+1}(X,E)$.  In local trivializations of $E$, $d_{A}$ is of the form $d+a$ for some $\mathfrak{g}$-valued $1$-form $\a\in\Om^{1}(X,\mathfrak{g}_{E})$. The curvature of $d_{A}$ is a $\mathfrak{g}$-valued $2$-form $F_{A}$ which is equal to $d_{A}^{2}$. The connection on $E$ induces one on $\mathfrak{g}_{E}$ and we will also denote the connection on $\mathfrak{g}_{E}$ by $d_{A}$. The Yang-Mills energy functional is
$$ YM(A):=\int_{X}Tr(F_{A}\wedge\ast F_{A})=\|F_{A}\|^{2}_{L^{2}(X)}.$$
A connection is called a Yang-Mills connection if it is a critical point of the Yang-Mills functional, i.e., $d_{A}^{\ast}F_{A}=0$ \cite[Euqation (2.1.34)]{Donaldson/Kronheimer}.   In addition, all connections satisfy the Bianchi identity $d_{A}F_{A}=0$ \cite[Euqation (2.1.21)]{Donaldson/Kronheimer}.  In fact, the connection $d_{A}$ on Bianchi identity is the connection on $\mathfrak{g}_{E}\subset\operatorname{End}E$. This implies that the Yang-Mills connection is a harmonic-form of degree $2$ on $End E$. 

We now focus on $\Om^{2}(X,E)$ which is the space of $2$-from on $E$. The spaces $\Om^{2}(X,E)$ can decomposed as
$$\Om^{2}(X,E)=\Om^{0,2}(X,E)\oplus\Om^{2,0}(X,E)\oplus P^{1,1}\oplus\Om^{1,1}_{0}(X,E),$$
where $\Om^{1,1}_{0}$ denotes the space of $(1,1)$-type propotional to $\w$. And let us now consider the operator $\sharp$ which defined in \cite{Suh}: 
\begin{equation}\nonumber
\sharp:\Om^{2}\xrightarrow{L^{n-2}/(n-2)!}\Om^{2n-2}\xrightarrow{\ast^{-1}=\ast}\Om^{2},\ i.e.,\ \sharp=\ast^{-1}\circ {L^{n-2}/(n-2)!}.
\end{equation}
Then we have the following results from the definition of $\sharp$ and Lemma \ref{L1}.
\begin{lemma}\label{L7}(\cite[Lemma 2.2]{Suh} )\\
(i) $P^{1,1}=\{\a\in\Om^{2}:\sharp\a=-\a \}$,\\
(ii) $\Om^{2,0}\oplus\Om^{0,2}=\{\a\in\Om^{2}:\sharp\a=\a\}$,\\
(iii) $\Om^{1,1}_{0}=\{\a\in\Om^{2}:\sharp\a=(n-1)\a\}$.	
\end{lemma}
Now we define an operator
$$
\tilde{\sharp}=\left\{
\begin{aligned}
&\sharp &   \Om^{2,0}\oplus\Om^{0,2}\oplus P^{1,1}\\
&\sharp/(n-1) &  \Om^{1,1}_{0}
\end{aligned}
\right.
$$
Then we get $\tilde{\sharp}^{2}=id$ which implies that $\Om^{2}$  decomposed into the self-dual part $\Om^{2}_{+}=\Om^{2,0}\oplus \Om^{0,2}\oplus\Om^{1,1}_{0}$ and the anti-self-dual part $P^{1,1}$. Hence
any $2$-form $\a$ can be splitted into the self-dual part $\a^{+}=\a^{2,0}+\a^{0,2}+\a_{0}\otimes\w$ and the anti-self-dual part $\a^{-}=\a^{1,1}_{0}$, where $\a_{0}=\frac{1}{n}\La\a$ and $\a^{1,1}_{0}=\a^{1,1}-\a_{0}$. Now we assert the following formula. Our proof here use Suh's arguments in \cite[Lemma 4.1]{Suh} for  Yang-Mills energy functional.
\begin{lemma}\label{L2}
If $\a\in\Om^{2}(X,E)$, then we have an identity:
$$tr(\a\wedge\ast\a)=-tr(\a\wedge\a)\wedge\frac{\w^{n-2}}{(n-2)!}+2|\a^{2,0}+\a^{0,2}|\frac{\w^{n}}{n!}+n|\a_{0}\otimes\w|^{2}\frac{\w^{n}}{n!}.$$
Furthermore, if $\a$ is in $L^{2}$, then
$$\|\a||^{2}=2\|\a^{0,2}+\a^{2,0}\|^{2}+n^{2}\|\a_{0}\|^{2}+\int_{X}tr(\a\wedge\a)\wedge\frac{\w^{n-2}}{(n-2)!}.$$
\end{lemma}
\begin{proof}
Since the $2$-form $\a$ is decomposed as $\a=\a^{2,0}+\a^{0,2}+\a_{0}\otimes\w+\a^{1,1}_{0}$, Lemma \ref{L7} yields $\sharp\a=\ast(\a\wedge\frac{\w^{n-2}}{(n-2)!})=\a^{2,0}+\a^{0,2}+(n-1)\a_{0}\otimes\w-\a_{0}^{1,1}$. Then we get
\begin{equation}\nonumber
\ast\a=(\a^{2,0}+\a^{0,2})\wedge\frac{\w^{n-2}}{(n-2)!}+\a_{0}\wedge\frac{\w^{n-1}}{(n-1)!}-\a_{0}^{1,1}\wedge\frac{\w^{n-2}}{(n-2)!}.
\end{equation}
By a direct calculation we have
\begin{equation}\label{E2.1}
\begin{split}
tr(\a\wedge\ast\a)&=tr(\a^{2,0}+\a^{0,2})\wedge(\a^{2,0}+\a^{0,2})\wedge\frac{\w^{n-2}}{(n-2)!}+tr(\a_{0}\wedge\a_{0})\frac{\w^{n}}{(n-1)!}\\
&-tr(a_{0}^{1,1}\wedge\a_{0}^{1,1})\wedge\frac{\w^{n-2}}{(n-2)!},
\end{split}
\end{equation}
\begin{equation}\label{E2.2}
\begin{split}
tr(\a\wedge\a)\wedge
\frac{\w^{n-2}}{(n-2)!}&=tr(\a^{2,0}+\a^{0,2})\wedge(\a^{2,0}+\a^{0,2})\wedge\frac{\w^{n-2}}{(n-2)!}+tr(\a_{0}\wedge\a_{0})\frac{\w^{n}}{(n-2)!}\\
&+tr(\a_{0}^{1,1}\wedge\a_{0}^{1,1})\wedge\frac{\w^{n-2}}{(n-2)!}.
\end{split}
\end{equation}
Thus, combining Equations (\ref{E2.1})--(\ref{E2.2}), we obtain Lemma \ref{L2}.
\end{proof}
\begin{proposition}\label{P1}
If $\a\in \Om^{2}(X,E)$ satisfies $(d_{A}+d^{\ast}_{A})\a=0$, then we have the following identities:\\
(i) $2\pa_{A}^{\ast}\a^{2,0}=-\sqrt{-1}n\pa_{A}\a_{0}$,\\
(ii) $2\bar{\pa}_{A}^{\ast}\a^{0,2}=\sqrt{-1}n\bar{\pa}_{A}\a_{0}$.
\end{proposition}
\begin{proof}
Our proof here use the author's arguments in \cite{Huang} Proposition 2.1 for  Yang-Mills connection. We  prove only the first identity, the second proof is similar. Recall the equation $d_{A}\a=0$, we take $(1,2)$ part, it implies that
\begin{equation}\nonumber
0=\bar{\pa}_{A}\a^{1,1}_{0}+\bar{\pa}_{A}(\a_{0}\otimes\w)+\pa_{A}\a^{0,2}.
\end{equation}
Thus, we obtain 
\begin{equation}\label{E2.3}
0=\sqrt{-1}\La(\bar{\pa}_{A}\a^{1,1}_{0})+\sqrt{-1}\La\bar{\pa}_{A}(\a_{0}\otimes\w)+\sqrt{-1}\La\pa_{A}\a^{0,2}.
\end{equation}
The equation $d^{\ast}_{A}\a=0$, we take $(0,1)$ part, it implies that
\begin{equation}\label{E2.4}
\pa^{\ast}_{A}(\a^{1,1}_{0}+\a_{0}\otimes\w)+\bar{\pa}^{\ast}_{A}\a^{0,2}=0.
\end{equation}
By K\"{a}hler identities, See Proposition \ref{P3}, we can write (\ref{E2.4}) to
\begin{equation}\label{E2.5}
\begin{split}
0&=\sqrt{-1}[\La,\bar{\pa}_{A}](\a^{1,1}_{0}+\a_{0}\otimes\w)-\sqrt{-1}[\La,\pa_{A}]\a^{0,2}\\
&=\sqrt{-1}\La(\bar{\pa}_{A}\a^{1,1}_{0})+\sqrt{-1}\La\bar{\pa}_{A}(\a_{0}\otimes\w)
-\sqrt{-1}n\bar{\pa}_{A}\a_{0}-\sqrt{-1}\La\pa_{A}\a^{0,2}.
\end{split}
\end{equation}
Thus, combining Equations (\ref{E2.3})--(\ref{E2.5}), we obtain
\begin{equation}\nonumber
\begin{split}
0&=-2\sqrt{-1}\La\pa_{A}\a^{0,2}_{A}-\sqrt{-1}n\bar{\pa}_{A}\a_{0}\\
&=2\bar{\pa}^{\ast}_{A}\a^{0,2}_{A}-\sqrt{-1}n\bar{\pa}_{A}\a_{0}.\\
\end{split}
\end{equation}
The second identity can be proved analogously.
\end{proof}
The curvature $F_{A}$ of a Hermitian connection $d_{A}$ on $E$ over a complete complex manifold is in $\Om^{2}(X,\mathfrak{g}_{E})$. Thus, we can decompose $F_{A}$ as
$$F_{A}=F^{2,0}_{A}+F^{1,1}_{A0}+\frac{1}{n}\hat{F}_{A}\otimes\w+F^{0,2}_{A},$$
where $\hat{F}_{A}:=\La F_{A}$ and $F^{1,1}_{A0}=F^{1,1}_{A}-\frac{1}{n}\hat{F}_{A}\otimes\w$. Following Lemma \ref{L2}, we can write Yang-Mills functional as
\begin{equation}\nonumber
\begin{split}
YM(A)&=4\|F^{0,2}_{A}\|^{2}+\|\La F_{A}\|^{2}-\int_{X}Tr(F_{A}\wedge F_{A})\wedge\frac{\w^{n-2}}{(n-2)!}.\\
\end{split}
\end{equation}
The energy functional $\|\La F_{A}\|^{2}$ plays an important role in the study of Hermitian-Einstein connections (See \cite{Don,Donaldson/Kronheimer,UY}). Following Proposition \ref{P1}, we have
\begin{proposition}\label{P2.7}
If $d_{A}$ is a Yang-Mills connection on a complete K\"{a}hler manifold. Then
\begin{equation}
2\bar{\pa}^{\ast}_{A}F^{0,2}_{A}=\sqrt{-1}\bar{\pa}_{A}\hat{F}_{A},\ 2\pa_{A}^{\ast}F^{2,0}_{A}=-\sqrt{-1}\pa_{A}\hat{F}_{A}.
\end{equation}
In particular, if $d_{A}\in\mathcal{A}^{1,1}_{E}$, then $d_{A}\hat{F}_{A}=0$, i.e., $\hat{F}_{A}$ is harmonic with respect to $\De_{A}$.
\end{proposition}
\begin{corollary}
Let $(X,\w)$ be a complete, K\"{a}hler manifold, $\dim_{\C}X=n$, with a $d$(bounded) K\"{a}hler form $\w$, $E$ be a Hermitian vector bundle over $X$ and $d_{A}$ be a Hermitian connection on $E$. If $d_{A}\in\mathcal{A}^{1,1}_{E}$ is a smooth Yang-Mills connection with $L^{2}$-curvature $F_{A}$, then\\
(i)  $d_{A}$ is a flat connection, for $n\geq3$,\\
(ii) $d_{A}$ is an anti-self-dual connection, for $n=2$.
\end{corollary}
\begin{proof}
By Proposition \ref{P2.7}, we can see that $\hat{F}_{A}$ is $L^{2}$-harmonic section on $E$. The Weizenb\"{o}ck formula gives
$$0=\De_{A}\hat{F}_{A}=\na_{A}^{\ast}\na_{A}\hat{F}_{A}.$$
Thus we have $\na_{A}\hat{F}_{A}=0$. The Kato inequality $|\na\hat{F}_{A}|\leq \na|\hat{F}_{A}|$, See \cite[Equation (6.20)]{FU} implies that $\na|\hat{F}_{A}|=0$. By Theorem \ref{T4}, we get $\hat{F}_{A}=0$. Following Proposition \ref{P2}, we have
$$\int_{X}Tr(F_{A}\wedge F_{A})\wedge\w^{n-2}=0,\  n\geq 3.$$ 
By the Energy identity of Yang-Mills connection, we have the connection $d_{A}$ is flat for $n\geq3$ and $d_{A}$ is anti-self-dual for $n=2$.   We complete the proof of this corollary.
\end{proof}

\section{Hodge theory}
As we derive estimates in this section (and also following sections), there will be many constants which appear. Sometimes we will take care to bound the size of these constants, but we will also use the following notation whenever the value of the constants are unimportant. We write $\a\lesssim\b$ to mean that $\a\leq C\b$ for some positive constant $C$ independent of certain parameters on which $\a$ and $\b$ depend. The parameters on which $C$ is independent will be clear or specified at each occurrence. We also use $\b\lesssim\a$ and $\a\approx\b$ analogously.
\subsection{Harmonic form with respect to $\De_{A}$}
Let $(X,g)$ be an oriented, smooth, Riemannian manifold, $\dim_{\mathbb{R}}X=n$, and $E$ be a Hermitian vector bundle over $X$. Assume now that $d_{A}$ is a Hermitian connection on $E$. The formal adjoint operator of $d_{A}$ acting on $\Om^{p}(X,E)$ is $d^{\ast}_{A}=(-1)^{np+1}\ast d_{A}\ast$, where the operator $\ast:\Om^{p}(X,E)\rightarrow\Om^{n-p}(X,E)$ induced by the Hodge-Poincar\'{e}-De Rahm operator $\ast_{g}$ \cite[Charp Vi, Section 3]{Dem}. Indeed, if $\a\in\Om^{p}(X,E)$, $\b\in\Om^{p+1}(X,E)$ have compact support, we get
$$\int_{X}\langle d_{A}\a,\b\rangle=\int_{X}\langle \a,d^{\ast}_{A}\b\rangle.$$
The Laplace-Beltrami operator associated to $d_{A}$ is the second order operator $\De_{A}=d_{A}d^{\ast}_{A}+d^{\ast}_{A}d_{A}$. 
\begin{definition}\label{D3.1}
The space of $L^{2}$-harmonic forms of degree of $p$ respect to the Laplace-Beltrami operator $\De_{A}$ is defined by
$$\mathcal{H}^{p}_{(2)}(X,E)=\{\a\in\Om^{p}_{(2)}(X,E):\De_{A}\a=0\}.$$
\end{definition}
In this article, we follow the method of Gromov's \cite{Gro} to choose a sequence of cutoff functions $\{f_{\varepsilon}\}$ satisfying the following conditions:\\
(i) $f_{\varepsilon}$ is smooth and takes values in the interval $[0,1]$, furthermore, $f_{\varepsilon}$ has compact support.\\
(ii) The subsets $f^{-1}_{\varepsilon}\subset X$, i.e., of the points $x\in X$ where $f_{\varepsilon}(x)=1$ exhaust $X$ as $\varepsilon\rightarrow0$.\\
(iii) The differential of $f_{\varepsilon}$ everywhere bounded by $\varepsilon$,
$$\|df_{\varepsilon}\|_{L^{\infty}(X)}=\sup_{x\in X}|df_{\varepsilon}|\leq\varepsilon.$$
Then we have an useful lemma as follows.
\begin{lemma}\label{L3.2}
If $\a$ is a $L^{2}$-harmonic form of degree $p$ with respect to  $\De_{A}$ over a complete Riemannian manifold $X$, then $\a$ also satisfies $$d_{A}\a=d_{A}^{\ast}\a=0.$$
\end{lemma}
\begin{proof}
By a simple computation
\begin{equation}\nonumber
\begin{split}
0&=\langle\De_{A}\a, f^{2}_{\varepsilon}\a\rangle_{L^{2}(X)}
=\langle d_{A}\a,d_{A}(f^{2}_{\varepsilon}\a)\rangle_{L^{2}(X)}+\langle d_{A}^{\ast}\a,d_{A}^{\ast}(f^{2}_{\varepsilon}\a)\rangle_{L^{2}(X)}\\
&=\langle d_{A}\a,f_{\varepsilon}df_{\varepsilon}\wedge\a+f^{2}_{\varepsilon}d_{A}\a\rangle_{L^{2}(X)}
+\langle d_{A}^{\ast}\a,\ast(f_{\varepsilon}df_{\varepsilon}\wedge\ast\a)+f^{2}_{\varepsilon}d_{A}^{\ast}\a\rangle_{L^{2}(X)}\\
&=I_{1}(\varepsilon)+I_{2}(\varepsilon),\\
\end{split}
\end{equation}
where
$$I_{1}(\varepsilon)=\int_{X}f^{2}_{\varepsilon}(|d_{A}\a|^{2}+|d_{A}^{\ast}\a|^{2})$$
and
$$I_{2}(\varepsilon)\leq2\int_{X}|df_{\varepsilon}||f_{\varepsilon}||\a|(|d_{A}\a|+|d^{\ast}_{A}\a|).$$
Then we choose $f_{\varepsilon}$, such that $|df_{\varepsilon}|^{2}<\varepsilon f_{\varepsilon}$ on $X$ and estimate $I_{2}$ by Schwartz inequality. This yields
$$I_{2}(\varepsilon)\leq2\varepsilon\|f_{\varepsilon}\a\|_{L^{2}(X)}\big{(}\int_{X}f_{\varepsilon}^{2}(|d_{A}\a|^{2}+|d_{A}^{\ast}\a|^{2})\big{)}^{1/2}$$
and hence $I_{1}\rightarrow0$ for $\varepsilon\rightarrow0$.
\end{proof}
\begin{remark}
We denote by $\Om^{k}_{(0)}(X,E)$ by the space of $C^{\infty}$ $k$-forms on $E$ with compact support on $X$. Then, for any $\a\in\Om^{k}_{0}(X,E)$, we have the identity
$$\langle \De_{A}\a,\a\rangle_{L^{2}(X)}=\|d_{A}\a\|^{2}_{L^{2}(X)}+\|d^{\ast}_{A}\a\|^{2}_{L^{2}(X)}.$$
The above identity also holds in the case of $\a\in\Om^{k}_{(2)}(X,E)$ such that $\De_{A}\a\in L^{2}$.
\end{remark} 
There are several commutation relations between the basic operators associated to a K\"{a}hler manifold $X$, all following more or less directly from the K\"{a}hler condition $d\w=0$; taken together, these are referred to as the K\"{a}hler identities \cite{Dem,Huy}. 
\begin{proposition}\label{P3}
Let $X$ be a complete K\"{a}hler manifold, $E$ a Hermitian vector bundle over $X$ and $d_{A}$ be a Hermitian connection on $E$. We have the following identities\\
(i)\ $[\La,\bar{\pa}_{A}]=-\sqrt{-1}\pa^{\ast}_{A}$,\ $[\La,\pa_{A}]=\sqrt{-1}\bar{\pa}^{\ast}_{A}$.\\
(ii)\ $[\bar{\pa}^{\ast}_{A},L]=\sqrt{-1}\pa_{A}$,\ $[\pa^{\ast}_{A},L]=-\sqrt{-1}\bar{\pa}_{A}$.
\end{proposition}
Since $\w$ is parallel, the operator $L^{k}:\Om^{p}(X,E)\rightarrow\Om^{p+2k}(X,E)$ defined by $L^{k}(\a)=\a\wedge\w^{k}$ for all $p$-forms commutes with $d_{A}$. But the operator $L^{k}$ does not commute with $d^{\ast}_{A}$ in general, therefore the operator $L^{k}$ does not commute with $\De_{A}$. 

If $A$ and $B$ are operators on forms, define the (graded) commutator as 
$$[A,B]=AB-(-1)^{degA\cdot degB}BA,$$
where $degT$ is the integer $d$ for $T$: $\oplus_{p+q=r}\Om^{p,q}(X,E)\rightarrow\oplus_{p+q=r+d}\Om^{p,q}(X,E)$. If $C$ is another endomorphism of degree $c$, the following $Jacobi$ $identity$ is easy to check
$$(-1)^{ca}\big{[}A,[B,C]\big{]}+(-1)^{ab}\big{[}B,[C,A]\big{]}+(-1)^{bc}\big{[}C,[A,B]\big{]}=0.$$
For any $(p,q)$-form $\a$ on $E$ which also satisfies $d_{A}\a=d_{A}^{\ast}\a=0$, we observe the following useful lemma. 
\begin{lemma}\label{L3.4}
If $\a\in\Om^{p,q}(X,E)\cap\mathcal{H}^{p+q}_{(2)}(X,E)$ on a complete K\"{a}hler manifold $X$, then
$$[d^{\ast}_{A},L^{k}]\a=0,\ \forall\ k\in\mathbb{N}.$$
In particular, $[\De_{A},L^{k}]\a=0$.
\end{lemma}
\begin{proof}
At first, following Lemma \ref{L3.2}, we have $d_{A}\a=(\bar{\pa}_{A}+\pa_{A})\a=0$, i.e.,  $$\bar{\pa}_{A}\a=\pa_{A}\a=0.$$ 
Then for any $k\in\mathbb{N}$, 
$$[\pa_{A},L^{k}]\a=0,\  \  [\bar{\pa}_{A},L^{k}]\a=0.$$
From K\"{a}hler identities, one has
\begin{equation}\nonumber
\begin{split}
[d^{\ast}_{A},L^{k}]\a&=d^{\ast}_{A}L^{k}\a=([d^{\ast}_{A},L]+Ld^{\ast}_{A})(L^{k-1}\a)\\
&=(\sqrt{-1}(\pa_{A}-\bar{\pa}_{A})+Ld^{\ast}_{A})(L^{k-1}\a)\\
&=Ld^{\ast}_{A}L^{k-1}\a=L([d_{A}^{\ast},L]+Ld^{\ast}_{A})(L^{k-2}\a)\\
&=L^{2}d^{\ast}_{A}L^{k-2}\a=\ldots=L^{k}(d_{A}^{\ast}\a)=0.
\end{split}
\end{equation}
Following the Jacobi identity, we have
$$[\De_{A},L^{k}]=[[d_{A},d_{A}^{\ast}],L^{k}]=[[L^{k},d_{A}],d_{A}^{\ast}]+[[d_{A}^{\ast},L^{k}],d_{A}]=0.$$
We complete this proof.
\end{proof}
\begin{remark}
Let $\a_{k}$ be a smooth $k$-form on $E$. Then we can write $\a_{k}=\sum_{p+q=k}a_{p,q}$. If $d_{A}\a_{k}=d_{A}^{\ast}\a_{k}=0$, then we have
\begin{equation}\label{E14}
\bar{\pa}_{A}\a_{p,q}+{\pa}_{A}\a_{p-1,q+1}=0,\ \bar{\pa}^{\ast}_{A}\a_{p,q}+\pa^{\ast}_{A}\a_{p+1,q-1}=0.
\end{equation}
If $d_{A}$ is a smooth Hermitian connection, then we always can not deduce from Equations (\ref{E14}) to the identity  $\bar{\pa}_{A}\a_{p,q}=0$ for any $0\leq p\leq k$. It means that $\bar{\pa}_{A}\a_{k}=0$ may be incorrect. One also can see Proposition \ref{P1} on $2$-forms case.
\end{remark}
Next, let us indicate one simple proposition:
\begin{proposition}\label{P3.5}
If the $L^{2}$-harmonic form $\a$ in $\Om^{k}_{(2)}(X,E)$ is $d_{A}(L^{2})$, i.e., there exists a $L^{2}$-form $\b$ such that $\a=d_{A}\b$, then $\a=0$.
\end{proposition}
\begin{proof}
Let $\{f_{\varepsilon}\}$  be a sequence of cutoff functions on \cite{Gro}. Noting that $d_{A}^{\ast}\a=0$ and $f_{\varepsilon}\a$ has compact support, one have
\begin{equation}\nonumber
\langle d_{A}\b,f_{\varepsilon}\a\rangle_{L^{2}(X)}=\langle\b, df_{\varepsilon}\wedge\a\rangle_{L^{2}(X)}.
\end{equation}
Then, we choose $|df_{\varepsilon}|^{2}<\varepsilon f_{\varepsilon}$ on $X$. This yields
\begin{equation}\nonumber
\begin{split}
|\langle d_{A}\b,f_{\varepsilon}\a\rangle_{L^{2}(X)}&=|\langle\b, df_{\varepsilon}\wedge\a\rangle_{L^{2}(X)}|\\
&\leq \varepsilon\|\a\|_{L^{2}(X)}\cdot\|\b\|_{L^{2}(X)}\\
&\rightarrow0,\ for\ \varepsilon\rightarrow0.
\end{split}
\end{equation}
Since $\lim_{\varepsilon\rightarrow0}f_{\varepsilon}\a(x)=\a(x)$, it follows from the dominated convergence theorem that
\begin{equation}\nonumber
\lim_{\varepsilon\rightarrow0}\langle d_{A}\b,f_{\varepsilon}\a\rangle_{L^{2}(X)}=\|\a\|^{2}_{L^{2}(X)}.
\end{equation}
Combining the preceding inequalities yields $\a=0$.
\end{proof}
\begin{proof}[\textbf{Proof of Theorem \ref{T1}}]
Let $k=p+q$. We denote $\b=\theta\wedge\w^{k-1}\wedge\a$. It's easy to see  $$d_{A}\b=d_{A}(\theta\wedge\w^{k-1}\wedge\a)=L^{k}\a.$$
  Since $\a$ is in $L^{2}$ and $\theta\wedge\w^{k-1}$ is bounded, 
$$\|\b\|_{L^{2}(X)}\leq\|\a\|_{L^{2}(X)}\|\theta\wedge\w^{k-1}\|_{L^{\infty}(X)},$$
i.e., $\b$ is also in $L^{2}$.Thus the form $L^{k}\a=\w^{k}\wedge\a$ is $d_{A}(L^{2})$

Furthermore, if $\a$ is harmonic, then following Lemma \ref{L3.4}, $L^{k}\a$ is harmonic with respect to $\De_{A}$. Thus from Proposition \ref{P3.5}, we have $L^{k}\a=0$
for any $k>0$.  Following Lemma \ref{L1}, it implies that $\a=0$, unless $p+q=n$.
\end{proof}
We also like to  give a lower bound on the spectra of the Laplace operator $\De_{A}$ on $L^{2}$-forms $\Om^{p}$ for $p\neq n$ to sharpen the Lefschetz vanishing theorem \ref{T1}. In \cite{Gro}, Gromov used the property $L^{k}$ commuted with $\De$ to give a lower bound. Although we can not get a similar result for all $L^{2}$-forms, the lower bound on the spectra of the Laplace operator $\De_{A}$  will be proved for the primitive $L^{2}$-forms.
\begin{proposition}
Let $(X,\w)$ be a complete, K\"{a}hler manifold, $\dim_{\C}X=n$, with a $d$(bounded) K\"{a}hler form $\w$, i.e. there exists a bounded $1$-form $\theta$ such that $\w=d\theta$. Then for any $L^{2}$-form $\a\in P^{i,j}:=P^{k}\cap\Om^{i,j}(X,E)$ on $X$ of degree $k:=i+j\neq n$ satisfies the inequality
\begin{equation}\label{E8}
c_{n,k}\|\theta\|^{-2}_{L^{\infty}(X)}\|\a\|^{2}_{L^{2}(X)}\leq\langle\a,\De_{A}\a\rangle_{L^{2}(X)},
\end{equation}
where $c_{n,k}>0$ is a constant which depends only on $n,k$.
\end{proposition}
\begin{proof}
Inequality (\ref{E8}) makes sense, strictly speaking, if $\De_{A}\a$ (as well as $\a$) is in $L^{2}$. In this case $(d_{A}+d^{\ast}_{A})\a$ is also  in $L^{2}$ by the proof
of Lemma \ref{L3.2} and (\ref{E8}) is equivalent to
\begin{equation}\label{E+}
\|(d_{A}+d_{A}^{\ast})\a\|_{L^{2}(X)}\geq c_{0}\|\a\|_{L^{2}(X)}.
\end{equation}
Moreover the cutoff argument in this section shows that the general case of (\ref{E+}), where we only assume $\a$ and $(d_{A}+d_{A}^{\ast})\a$ in $L^{2}$, follows from that where $\a$ is a smooth function with compact support. In particular, inequality (\ref{E8}) with $\a$ in $L^{2}$ implies the general case of (\ref{E+}).

The linear map $L^{n-k}:\Om^{k}\rightarrow\Om^{2n-k}$ for $k\leq n-1$ is a bijective quasi-isometry on $P^{i,j}$ $(i+j=k)$, thus any $\a\in P^{i,j}$ satisfies
$$\a=C(n,k)\ast L^{n-k}\a=C(n,k)\ast(\a\wedge\w^{n-k}),$$
where $C(n,k)=\sqrt{-1}^{j^{2}-i^{2}}(-1)^{ij}\frac{1}{(n-k)!}$. We write $\ast\a=d_{A}\eta-\tilde{\a}$, for $$\eta=C(n,k)(\theta\wedge\a\wedge\w^{n-k-1}),\ \tilde{\a}=C(n,k)(\theta\wedge d_{A}\a\wedge\w^{n-k-1})$$
Noting that
$$\|\a\|_{L^{2}(X)}=\|\ast\a\|_{L^{2}(X)},\ \|\eta\|_{L^{2}(X)}\lesssim\|\theta\|_{L^{\infty}(X)}\|\a\|_{L^{2}(X)}.$$
We then have
$$\|\tilde{\a}\|_{L^{2}(X)}\lesssim \|\theta\|_{L^{\infty}(X)}\|d_{A}\a\|^{2}_{L^{2}(X)} \lesssim\|\theta\|_{L^{\infty}(X)}\langle\De_{A}\a,\a\rangle_{L^{2}(X)}^{1/2}.$$
We observe that
\begin{equation}\nonumber
\begin{split}
|\langle\ast\a,d_{A}\eta\rangle_{L^{2}(X)}|&=|\langle\ast d_{A}\a,\eta\rangle_{L^{2}(X)}|\\
&\leq\|d_{A}\a\|_{L^{2}(X)}\|\eta\|_{L^{2}(X)}\\
&\leq\langle\De_{A}\a,\a\rangle_{L^{2}(X)}^{1/2}\|\eta\|_{L^{2}(X)}\\
&\lesssim\langle\De_{A}\a,\a\rangle_{L^{2}(X)}^{1/2}\|\theta\|_{L^{\infty}(X)}\|\a\|_{L^{2}(X)}\\
\end{split}
\end{equation}
and
\begin{equation}\nonumber
|\langle\a,\tilde{\a}\rangle|_{L^{2}(X)}\leq\|\a\|_{L^{2}(X)}\|\tilde{\a}\|_{L^{2}(X)}\lesssim\langle\De_{A}\a,\a\rangle_{L^{2}(X)}^{1/2}\|\theta\|_{L^{\infty}(X)}\|\a\|_{L^{2}(X)}.
\end{equation}
It now follows from above inequalities that 
\begin{equation}\nonumber
\begin{split}
\|\a\|_{L^{2}(X)}&=\|\ast\a\|_{L^{2}(X)}\\
&\leq |\langle\ast\a,d_{A}\eta\rangle_{L^{2}(X)}|+|\langle\ast\a,\tilde{\a}\rangle|_{L^{2}(X)}\\
&\lesssim\|\theta\|_{L^{\infty}(X)}\langle\De_{A}\a,\a\rangle_{L^{2}(X)}^{1/2}.\\
\end{split}
\end{equation}
The case $k>n$ follows by the Poincare duality as the operator $\ast:\Om^{k}\rightarrow\Om^{2n-k}$ commutes with $\De_{A}$.
\end{proof}
\subsection{Uniform positive lower bounds for the least eigenvalue of $\De_{A}$}
There are several commutation relations between the basic operators associated to a K\"{a}hler manifold $(X,\w)$, all following more or less directly from the K\"{a}hler condition $d\w=0$. Taken together, these are referred to as the K\"{a}hler identities. At first, we observe that the operator $L^{k}$ commutes with $\De_{A}$ for any connection $d_{A}\in\mathcal{A}_{E}^{1,1}$.
\begin{lemma}\label{L3}
$$[\De_{A},L^{k}]=2k\sqrt{-1}(F^{2,0}_{A}-F^{0,2}_{A})L^{k-1},\ \forall\ k\in\mathbb{N}.$$
In particular, if the connection $d_{A}\in\mathcal{A}_{E}^{1,1}$, then $\De_{A}$ commutes with $L^{k}$ for any $k\in\mathbb{N}$.
\end{lemma}
\begin{proof}
The case of $k=1$: the operators $d_{A}$,\ $d^{\ast}_{A}$ and $L$ satisfy the following Jacobi identity:
$$-\big{[}L,[d_{A},d^{\ast}_{A}]\big{]}+\big{[}d^{\ast}_{A},[L,d_{A}]\big{]}+\big{[}d_{A},[d^{\ast}_{A},L]\big{]}=0.$$
Then we have
\begin{equation}\nonumber
\begin{split}
[L,\De_{A}]&=\big{[}d_{A},[d^{\ast}_{A},L]\big{]}\\
&=[\pa_{A}+\bar{\pa}_{A},\sqrt{-1}(\pa_{A}-\bar{\pa}_{A})]\\
&=[\sqrt{-1}\pa_{A},\pa_{A}]-[\sqrt{-1}\bar{\pa}_{A},\bar{\pa}_{A}]\\
&=2\sqrt{-1}(F^{2,0}_{A}-F^{0,2}_{A}).\\
\end{split}
\end{equation}
We suppose that the case of $p=k-1$ is true, i.e., $$[\De_{A},L^{k-1}]=2(k-1)\sqrt{-1}(F^{2,0}_{A}-F^{0,2}_{A})L^{k-2}.$$ Thus if $p=k$, we have
\begin{equation}\nonumber
\begin{split}
[\De_{A},L^{k}]&=[\De_{A},L]L^{k-1}+L[\De_{A},L^{k-1}]\\
&=2\sqrt{-1}(F^{0,2}_{A}-F^{2,0}_{A})L^{k-1}+2(k-1)\sqrt{-1}L(F^{2,0}_{A}-F^{0,2}_{A})L^{k-2}\\
&=2k\sqrt{-1}(F^{2,0}_{A}-F^{0,2}_{A})L^{k-1}.
\end{split}
\end{equation}
If $d_{A}\in\mathcal{A}_{E}^{1,1}$, then $[\De_{A},L^{k}]=0$.
\end{proof}
\begin{proof}[\textbf{Proof of Theorem \ref{T3}}] \ \\
\noindent \textbf{Case I: $d_{A}\in\mathcal{A}_{E}^{1,1}$}\\
Let $\a$ be a $p$-form on vector bundle, we denote $\b=L^{k}\a=\w^{k}\wedge\a$. We recall the operator $L^{k}:\Om^{p}(X,E)\rightarrow\Om^{p+2k}(X,E)$ for a given $p<n$ and $p+k=n$. Since the Lefschetz theorem $L^{k}$ is a bijective quasi-isometry,  
$$\|\a\|_{L^{2}(X)}\thickapprox\|\b\|_{L^{2}(X)}.$$ 
If $\a$ is in $L^{2}$, $\b$ is also in $L^{2}$. Since $F_{A}^{0,2}=0$, following Lemma \ref{L3}, $[L^{k},\De_{A}]=0$. Then we have
$$\langle\De_{A}\b,\b\rangle_{L^{2}(X)}=\langle L^{k}(\De_{A}\a)\,L^{k}\a\rangle_{L^{2}(X)}\approx\langle\De_{A}\a,\a\rangle_{L^{2}(X)}.$$
We write $\b=d_{A}\eta-\tilde{\a}$, for $\eta=\a\wedge\w^{k-1}\wedge\theta$ and $\tilde{\a}=d_{A}\a\wedge\w^{k-1}\wedge\theta$. Observe that
\begin{equation}\label{E01}
\|\eta\|_{L^{2}(X)}\lesssim\|\theta\|_{L^{\infty}(X)}\|\a\|_{L^{2}(X)}\lesssim\|\theta\|_{L^{\infty}(X)}\|\b\|_{L^{2}(X)},
\end{equation}
and
\begin{equation}\label{E02}
\begin{split}
\|\tilde{\a}\|_{L^{2}(X)}&\lesssim\|d_{A}\a\|_{L^{2}(X)}\|\theta\|_{L^{\infty}(X)}\lesssim\langle\De_{A}\a,\a\rangle_{L^{2}(X)}^{1/2}\|\theta\|_{L^{\infty}(X)}.\\
\end{split}
\end{equation}
We then have
\begin{equation}\nonumber
\begin{split}
\|\b\|^{2}_{L^{2}(X)}&\leq|\langle\b,d_{A}\eta\rangle_{L^{2}(X)}|+|\langle\b,\tilde{\a}\rangle_{L^{2}(X)}|\\
&\leq|\langle d_{A}^{\ast}\b,\eta\rangle_{L^{2}(X)}|+|\langle\b,\tilde{\a}\rangle_{L^{2}(X)}|\\
&\lesssim\langle\De_{A}\b,\b\rangle_{L^{2}(X)}^{1/2}\|\theta\|_{L^{\infty}(X)}\|\b\|_{L^{2}(X)}
+\|\b\|_{L^{2}(X)}\|d_{A}\a\|_{L^{2}(X)}\|\theta\|_{L^{\infty}(X)}\\
&\lesssim\langle\De_{A}\a,\a\rangle_{L^{2}(X)}^{1/2}\|\theta\|_{L^{\infty}(X)}\|\b\|_{L^{2}(X)}.\\
\end{split}
\end{equation}
This yields the desired estimation
$$
\|\a\|^{2}_{L^{2}(X)}\lesssim\|\b\|^{2}_{L^{2}(X)}
\lesssim\|\theta\|^{2}_{L^{\infty}(X)}\langle\De_{A}\a,\a\rangle_{L^{2}(X)}.$$
\textbf{Case II: $d_{A}\notin\mathcal{A}^{1,1}_{E}$}\\
Now, we begin to  consider the case which the $(0,2)$-part of the curvature is non-zero. We denote by $\a$ a $L^{2}$ harmonic $p$-form, $p<n$ and $\b=L^{k}\a$, $k=n-p$, thus 
\begin{equation}\label{E04}
\|\a\|_{L^{2}(X)}\thickapprox\|\b\|_{L^{2}(X)},\  \langle\De_{A}\a,\a\rangle_{L^{2}(X)}\thickapprox\langle L^{k}(\De_{A}\a),L^{k}\a\rangle_{L^{2}(X)}.
\end{equation} 
Following Lemma \ref{L3}, we obtain
\begin{equation}\nonumber
\begin{split}
\langle\De_{A}\b,\b\rangle_{L^{2}(X)}&=\langle[\De_{A},L^{k}]\a+L^{k}(\De_{A}\a),L^{k}\a\rangle_{L^{2}(X)}\\
&=\langle2k\sqrt{-1}(F^{2,0}_{A}-F^{0,2}_{A})L^{k-1}\a,L^{k}\a\rangle_{L^{2}(X)}+\langle L^{k}(\De_{A}\a),L^{k}\a\rangle_{L^{2}(X)}.\\
\end{split}
\end{equation}
Thus we have
\begin{equation}\label{E03}
\begin{split}
\langle\De_{A}\b,\b\rangle_{L^{2}(X)}
&\lesssim\int_{X}(|F_{A}^{0,2}|+|F_{A}^{2,0}|)|\a|^{2}+\langle L^{k}(\De_{A}\a),L^{k}\a\rangle_{L^{2}(X)}\\
&\lesssim\langle\De_{A}\a,\a\rangle_{L^{2}(X)}+\|F_{A}^{0,2}\|_{L^{n}(X)}\|\a\|^{2}_{L^{2n/n-1}(X)}\\
&\lesssim\langle\De_{A}\a,\a\rangle_{L^{2}(X)}+\|F_{A}^{0,2}\|_{L^{n}(X)}\|\a\|^{2}_{L^{2}_{1}(X)}.\\
\end{split}
\end{equation}
Following in the case I, we can write $\b=d_{A}\eta-\tilde{\a}$,\ for $\eta=\a\wedge\w^{k-1}\wedge\theta$ and $\tilde{\a}=d_{A}\a\wedge\w^{k-1}\wedge\theta$. Combining the Equations (\ref{E01}) and (\ref{E02}) with the estimates (\ref{E04})--(\ref{E03}) yields
\begin{equation}\label{E06}
\begin{split}
\|\b\|^{2}_{L^{2}(X)}&\leq |\langle\b,d_{A}\eta\rangle_{L^{2}(X)}|+|\langle\b,\tilde{\a}\rangle_{L^{2}(X)}|\\
&=|\langle d_{A}^{\ast}\b,\eta\rangle_{L^{2}(X)}|+|\langle\b,\tilde{\a}\rangle_{L^{2}(X)}|\\
&\lesssim\langle\De_{A}\b,\b\rangle_{L^{2}(X)}^{1/2}\|\theta\|_{L^{\infty}(X)}\|\a\|_{L^{2}(X)}
+\langle\De_{A}\a,\a\rangle_{L^{2}(X)}^{1/2}\|\theta\|_{L^{\infty}(X)}\|\b\|_{L^{2}(X)}\\
&\lesssim\langle(\De_{A}\b,\b\rangle_{L^{2}(X)}^{1/2}
+\langle\De_{A}\a,\a\rangle_{L^{2}(X)}^{1/2})\|\theta\|_{L^{\infty}(X)}\|\a\|_{L^{2}(X)}\\
\end{split}
\end{equation}
This yields the desired estimation
\begin{equation}\label{E07}
\begin{split}
\|\a\|^{2}_{L^{2}(X)}&\lesssim\|\b\|^{2}_{L^{2}(X)}\lesssim\|\theta\|_{L^{\infty}(X)}\|\a\|_{L^{2}(X)}
(\langle\De_{A}\a,\a\rangle_{L^{2}(X)}^{1/2}+\|F_{A}^{0,2}\|^{1/2}_{L^{n}(X)}\|\a\|_{L^{2}_{1}(X)})\\
&\lesssim\|\theta\|_{L^{\infty}(X)}\|\a\|_{L^{2}(X)}(\langle\De_{A}\a,\a\rangle^{1/2}_{L^{2}(X)}
+\|F_{A}^{0,2}\|^{1/2}_{L^{n}(X)}(\langle\De_{A}\a,\a\rangle_{L^{2}(X)}^{1/2}+\|\a\|_{L^{2}(X)})\\
&\lesssim\|\theta\|_{L^{\infty}(X)}
\big{(}(1+\|F^{0,2}_{A}\|^{1/2}_{L^{n}(X)})\|\a\|_{L^{2}(X)}\langle\De_{A}\a,\a\rangle_{L^{2}(X)}^{1/2}
+\|F^{0,2}_{A}\|^{1/2}_{L^{n}(X)}\|\a\|^{2}_{L^{2}(X)}\big{)}.
\end{split}
\end{equation}
If we replace the $L^{n}$-norm of $F_{A}^{0,2}$ to $L^{\infty}$-norm, then inequalities (\ref{E03}) and (\ref{E07}) replaced to
	\begin{equation}\label{E03'}
	\langle\De_{A}\b,\b\rangle_{L^{2}(X)}\lesssim\langle\De_{A}\a,\a\rangle_{L^{2}(X)}+\|F_{A}^{0,2}\|_{L^{\infty}(X)}\|\a\|^{2}_{L^{2}(X)}.
	\end{equation}
	and
	\begin{equation}\label{E07'}
	\begin{split}
	\|\a\|^{2}_{L^{2}(X)}&\lesssim\|\b\|^{2}_{L^{2}(X)}\lesssim\|\theta\|_{L^{\infty}(X)}\|\a\|_{L^{2}(X)}
	(\langle\De_{A}\a,\a\rangle_{L^{2}(X)}^{1/2}+\|F_{A}^{0,2}\|^{1/2}_{L^{\infty}(X)}\|\a\|_{L^{2}(X)})\\
	&\lesssim\|\theta\|_{L^{\infty}(X)}\|\a\|_{L^{2}(X)}\langle\De_{A}\a,\a\rangle^{1/2}_{L^{2}(X)}
	+\|\theta\|_{L^{\infty}(X)}\|F_{A}^{0,2}\|^{1/2}_{L^{\infty}(X)}\|\a\|^{2}_{L^{2}(X)}.\\
	\end{split}
	\end{equation}
Thus, we can choose $\|F^{0,2}_{A}\|_{L^{n}(X)}$ or $\|F_{A}^{0,2}\|_{L^{\infty}(X)}$ small enough to ensure (\ref{E08}).
\end{proof}
If we assume that $E$ possesses a flat Hermitian connection $d_{A}$ which means that $F_{A}=0$, or equivalently, that $E$ is given by a representation $\pi_{X}\rightarrow U(r)$. Thanks to our flatness assumption, there are orthogonal decompositions (\cite[Chapter VIII Theorem 3.2]{Dem} )
$$\Om^{k}_{(2)}(X)\otimes E=\mathcal{H}^{k}_{(2)}(X,E)\oplus\overline{{\rm{Im}} d_{A}}\oplus\overline{{\rm{Im}} d^{\ast}_{A}}.$$
We denote by $H^{p}_{DR}(X,E)$ the $L^{2}$ de Rham cohomology groups, namely the cohomology groups of complex $(K^{\bullet},d_{A})$ defined by
$$K^{p}:=\{u\in\Om^{p}_{(2)}(X)\otimes E: d_{A}u=0\}.$$
In other words,\ one has $H_{DR}^{p}(X,E)=\ker{d_{A}}/{\rm {Im}}{d_{A}}$, where $d_{A}$ is a $L^{2}$-extension of the connection calculated in the sense of distributions.
\begin{theorem}
Let $(X,\w)$ be a complete, K\"{a}hler manifold, $\dim_{\C}X=n$, with a $d$(bounded) K\"{a}hler form $\w$, i.e. there exists a bounded $1$-form $\theta$ such that $\w=d\theta$. If $d_{A}$ is a flat connection on $E$ and $H_{DR}^{k}(X,E)$ ($k\neq n$) is finite dimensional, then
$$H^{k}_{DR}(X,E)=0.$$
\end{theorem}
We also can extend the vanishing Theorem \ref{T1} to the case of sublinear growing one form $\theta$. The method of this proof  follows the idea in \cite{CX}.
\begin{proof}[\textbf{Proof of Theorem \ref{T5}}]
By hypothesis, there exists a 1-form $\theta$ with $d\theta=\w$ and
$$\|\theta(x)\|_{L^{\infty}(X)}\leq c(1+\rho(x,x_{0})),$$
where $c$ is an absolute constant. In what follows we assume that the distance function $\rho(x, x_{0})$ is smooth for $x\neq x_{0}$. The general case follows easily by an approximation argument.
	
Let $\eta:\mathbb{R}\rightarrow\mathbb{R}$ be smooth, $0\leq\eta\leq1$,
$$\eta(t)=\left\{
\begin{aligned}
1, &  & t\leq0 \\
0,  &  & t\geq1
\end{aligned}
\right.
$$
and consider the compactly supported function
$$f_{j}(x)=\eta(\rho(x_{0},x)-j),$$
where $j$ is a positive integer.
	
Let $\a$ be a harmonic $k$-form in $\Om^{k}_{(2)}(X,E)$, and consider the form $\b=\theta\wedge\a$.  From Lemma \ref{L3}, we have  $\w\wedge\a\in \mathcal{H}_{(2)}^{k+2}(X,E)$, thus $d^{\ast}_{A}(\w\wedge\a)=0$. Noticing that $f_{j}\b$ has compact support, one has
$$0=\langle d^{\ast}_{A}(\w\wedge\a),f_{j}\b\rangle_{L^{2}(X)}=\langle\w\wedge\a,d_{A}(f_{j}\b)\rangle_{L^{2}(X)}.$$
We further note that, since $\w=d\theta$ and $d_{A}\a=0$,
\begin{equation}\label{E3}
\begin{split}
0&=\langle\w\wedge\a, d_{A}(f_{j}\b)\rangle_{L^{2}(X)}\\
&=\langle\w\wedge\a, f_{j}d_{A}\b\rangle_{L^{2}(X)}+\langle\w\wedge\a, df_{j}\wedge\b\rangle_{L^{2}(X)}\\
&=\langle\w\wedge\a, f_{j}\w\wedge\a\rangle_{L^{2}(X)}+\langle\w\wedge\a, df_{j}\wedge\b\rangle_{L^{2}(X)}\\
&=\langle\w\wedge\a, f_{j}\w\wedge\a\rangle_{L^{2}(X)}+\langle\w\wedge\a, df_{j}\wedge\theta\wedge\a\rangle_{L^{2}(X)}.\\
\end{split}
\end{equation}
Since $0\leq f_{j}\leq 1$ and $\lim_{j\rightarrow\infty}f_{j}(x)(\w\wedge\a)(x)=(\w\wedge\a)(x)$, it follows from the dominated convergence theorem that
\begin{equation}\label{E4}
\lim_{j\rightarrow\infty}\langle\w\wedge\a, f_{j}\w\wedge\a\rangle_{L^{2}(X)}=\|\w\wedge\a\|^{2}_{L^{2}(X)}.
\end{equation}
Since $\w$ is bounded, $supp(df_{j})\subset B_{j+1}\backslash B_{j}$ and $\|\theta(x)\|_{L^{\infty}}=O(\rho(x_{0},x))$, one obtains that
\begin{equation}\label{E2}
|\langle\w\wedge\a,df_{j}\wedge\theta\wedge\a\rangle_{L^{2}(X)}|\leq (j+1)C\int_{B_{j+1}\backslash B_{j}}|\a(x)|^{2}dx,
\end{equation}
where $C$ is a constant independent of $j$.
	
We claim that there exists a subsequence $\{j_{i}\}_{i\geq1}$ such that
\begin{equation}\label{E1}
\lim_{i\rightarrow\infty}(j_{i}+1)\int_{B_{j_{i}+1}\backslash B_{j_{i}}}|\a(x)|^{2}dx=0.
\end{equation}
If not, there exists a positive constant $a$ such that
$$\lim_{j\rightarrow\infty}(j+1)\int_{B_{j+1}\backslash B_{j}}|\a(x)|^{2}dx\geq a>0,\ j\geq1.$$
This inequality implies
\begin{equation}\nonumber
\begin{split}
\int_{X}|\a(x)|^{2}dx&=\sum_{j=0}^{\infty}\int_{B_{j+1}\backslash B_{j}}|\a(x)|^{2}dx\\
&\geq a\sum_{j=0}^{\infty}\frac{1}{j+1}\\
&=+\infty\\
\end{split}
\end{equation}
which is a contradiction to the assumption $\int_{X}|\a(x)|^{2}dx<\infty$.\ Hence,\ there exists a subsequence $\{j_{i}\}_{i\geq1}$ for which (\ref{E1}) holds.\ Using (\ref{E2}) and (\ref{E1}),\ one obtains
\begin{equation}\label{E5}
\lim_{i\rightarrow\infty}\langle\w\wedge\a, df_{j_{i}}\wedge\theta\wedge\a\rangle_{L^{2}(X)}=0
\end{equation}
It now follows from (\ref{E3}), (\ref{E4}) and (\ref{E5}) that $\w\wedge\a=0$. Following Lemma \ref{L1}, we have $\a=0$. 
\end{proof}
We will construct a priori estimate for any smooth form on vector bundle under the condition that $L^{n}$-norm of the curvature $F_{A}$ is small enough.
\begin{lemma}\label{L4}
Let $(X,\w)$ be a complete, K\"{a}hler manifold, $\dim_{\C}X=n$, with a $d$(bounded) K\"{a}hler form $\w$, $E$ be a Hermitian vector bundle over $X$ and $d_{A}$ be a Hermitian connection on $E$. Suppose that the Riemannian curvature tensor of the K\"{a}hler metric $g$ on $X$ is bounded. Then there are constants $\varepsilon=\varepsilon(X,n)\in(0,1]$ and $C=C(X,n)\in[1,\infty)$ with following significance. If the curvature $F_{A}$ of the connection $d_{A}$ obeys\\
(1) $\|F_{A}\|_{L^{n}(X)}\leq \varepsilon$ or\\
(2) $F_{A}$ has a bounded $L^{\infty}$-norm,\\
then any $\a\in\Om^{k}_{(2)}(X,E)$ satisfies
\begin{equation}\label{E3.7}
\|\na_{A}\a\|^{2}_{L^{2}(X)}\leq C(\langle\De_{A}\a,\a\rangle_{L^{2}(X)}+\|\a\|^{2}_{L^{2}(X)}).
\end{equation}
\end{lemma}
\begin{proof}
Inequality (\ref{E3.7}) makes sense, strictly speaking, if $\De_{A}\a$ (as well as $\a$) is in $L^{2}$. At first, the Weizenb\"{o}ck formulas gives,
$$\De_{A}\a=\na_{A}^{\ast}\na_{A}\a+\{Riem,\a\}+\{F_{A},\a\},\ \a\in\Om^{k}(X,E).$$
Let $\{y_{j}\}_{j=1}^{\infty}$ be a sequence of cutoff functions,
$0\leq y_{j}\leq1$, with $lim_{j\rightarrow\infty}y_{j}=1$ pointwise and $|dy_{j}|$ uniformly bounded. Then
Thus we have
\begin{equation}\label{E+2}
\langle \De_{A}\a,y_{i}^{2}\a\rangle_{L^{2}(X)} 
=\langle\na_{A}^{\ast}\na_{A}\a,y_{i}^{2}\a\rangle_{L^{2}(X)}+\langle\{Riem,\a\}+\{F_{A},\a\},y^{2}_{i}\a\rangle_{L^{2}(X)}
\end{equation}
The first term on left hand of (\ref{E+2}) : 
\begin{equation*}
\begin{split}   
\langle\na_{A}^{\ast}\na_{A}\a,y_{i}^{2}\a\rangle&=\langle\na_{A}\a,\na_{A}(y_{i}^{2}\a)\rangle_{L^{2}(X)}\\
&=\langle\na_{A}\a,y_{i}^{2}\na_{A}\a\rangle_{L^{2}(X)}+\langle\na_{A}\a,2y_{i}\na y_{i}\a\rangle_{L^{2}(X)}\\
&=\| y_{i}\na_{A}\a+\na y_{i}\a\|^{2}_{L^{2}(X)}-\|\na y_{i}\a\|^{2}_{L^{2}(X)}\\
&=\|\na_{A}(y_{i}\a)\|_{L^{2}(X)}^{2}-\|\na y_{i}\a\|^{2}_{L^{2}(X)}.\\
\end{split}
\end{equation*}
The second term on left hand of (\ref{E+2}):
\begin{equation}\nonumber
|\langle\{Riem,\a\},y_{i}^{2}\a\rangle|_{L^{2}(X)}\leq\sup|Riem|\cdot\|y_{i}\a\|^{2}_{L^{2}(X)}.
\end{equation}
The third term on left hand of (\ref{E+2}):
\begin{equation}\label{E09}
\|\langle\{F_{A},\a\},y_{i}^{2}\a\rangle\|_{L^{2}(X)}\lesssim\|F_{A}\|_{L^{n}(X)}\|y_{i}\a\|^{2}_{L^{\frac{2n}{n-1}}(X)}
\lesssim\|F_{A}\|_{L^{n}(X)}\|y_{i}\a\|^{2}_{L^{2}_{1}(X)},
\end{equation}
here we use the Sobolev embedding $L^{2}_{1}\hookrightarrow L^{2n/(n-1)}$.
This yields the desired estimation
\begin{equation}\label{E10}
\begin{split}
\|\na_{A}(y_{i}\a)\|^{2}_{L^{2}(X)}&\lesssim \langle\De_{A}\a,y_{i}^{2}\a\rangle_{L^{2}(X)}+\|y_{i}\a\|^{2}_{L^{2}(X)}+\|F_{A}\|_{L^{n}(X)}\|y_{i}\a\|^{2}_{L^{2}_{1}(X)}+\|\na y_{i}\a\|^{2}_{L^{2}(X)}\\
&\lesssim \langle\De_{A}\a,y_{i}^{2}\a\rangle_{L^{2}(X)}+(1+\|F_{A}\|_{L^{n}(X)})\|y_{i}\a\|^{2}_{L^{2}(X)}\\
&+\|F_{A}\|_{L^{n}(X)}\|\na_{A}(y_{i}\a)\|^{2}_{L^{2}(X)}+\|\na y_{i}\a\|^{2}_{L^{2}(X)}.\\
\end{split}
\end{equation}
We can choose $\|F_{A}\|_{L^{n}(X)}$ sufficiently small to ensure the inequality 
\begin{equation}\label{E11}
\|\na_{A}(y_{i}\a)\|^{2}_{L^{2}(X)}\lesssim\langle\De_{A}\a,y_{i}^{2}\a\rangle_{L^{2}(X)}+\|y_{i}\a\|^{2}_{L^{2}(X)}.
\end{equation}
Taking $i\rightarrow\infty$, we obtain inequality (\ref{E3.7}).

If we replace the $L^{n}$-norm of $F_{A}^{0,2}$ to $L^{\infty}$-norm, then inequalities (\ref{E09}) and (\ref{E10}) replaced to
$$\|\langle\{F_{A},\a\},y_{i}^{2}\a\rangle\|_{L^{2}(X)}\lesssim\|F_{A}\|_{L^{\infty}(X)}\|y_{i}\a\|^{2}_{L^{2}(X)},$$
and	
\begin{equation}\nonumber
\begin{split}
\|\na_{A}(y_{i}\a)\|^{2}_{L^{2}(X)}&\lesssim \langle\De_{A}\a,y_{i}^{2}\a\rangle_{L^{2}(X)}+\|y_{i}\a\|^{2}_{L^{2}(X)}+\|F_{A}\|_{L^{\infty}(X)}\|y_{i}\a\|^{2}_{L^{2}(X)}+\|\na y_{i}\a\|^{2}_{L^{2}(X)}\\
&\lesssim \langle\De_{A}\a,y_{i}^{2}\a\rangle_{L^{2}(X)}+(1+\|F_{A}\|_{L^{\infty}(X)})\|y_{i}\a\|^{2}_{L^{2}(X)}+\|\na y_{i}\a\|^{2}_{L^{2}(X)}.\\
\end{split}
\end{equation}
If $F_{A}$ has a bounded $L^{\infty}$-norm, then taking $i\rightarrow\infty$, we obtain inequality (\ref{E3.7}).	
\end{proof}
As we derive estimates in this section (and also following section), we assume that any smooth $L^{2}$ form $\psi$ has compact support. Otherwise we can replace $\psi$ to $y_{j}\psi$, where $y_{j}$ is the cutoff function on Lemma \ref{L4}.  
\begin{proof}[\textbf{Proof of Theorem \ref{T1.4}}]
Following the Bochner-Kodaira-Nakano formula \cite[Chapter VII., Corollary 1.3]{Dem} $$\De_{\bar{\pa}_{E}}=\De_{\pa_{E}}+\sqrt{-1}[\Theta(E),\La],$$
we have
$$\De_{E}=\De_{\bar{\pa}_{E}}+\De_{\pa_{E}}=2\De_{\bar{\pa}_{E}}-\sqrt{-1}[\Theta(E),\La],$$
where $\De_{E}:=D_{E}D_{E}^{\ast}+D_{E}^{\ast}D_{E}$. Then for any $\a\in\Om^{p,q}_{(2)}(X,E)$, we have
\begin{equation}\label{E12}
\begin{split}
\langle\De_{E}\a,\a\rangle_{L^{2}(X)}&\leq 2\langle\De_{\bar{\pa}_{E}}\a,\a\rangle_{L^{2}(X)}+|\langle[\Theta(E),\La]\a,\a\rangle_{L^{2}(X)}|\\
&\leq2\langle\De_{\bar{\pa}_{E}}\a,\a\rangle_{L^{2}(X)}+|\langle\Theta(E)(\La\a),\a\rangle_{L^{2}(X)}|+|\langle\Theta(E)\a,\a\wedge\w\rangle_{L^{2}(X)}|\\
&\lesssim\langle\De_{\bar{\pa}_{E}}\a,\a\rangle_{L^{2}(X)}+\|\Theta(E)\|_{L^{n}(X)}\|\a\|^{2}_{L^{\frac{2n}{n-1}}(X)}\\
&\lesssim\langle\De_{\bar{\pa}_{E}}\a,\a\rangle_{L^{2}(X)}+\|\Theta(E)\|_{L^{n}(X)}\|\a\|^{2}_{L^{2}_{1}(X)}\\
&\lesssim\langle\De_{\bar{\pa}_{E}}\a,\a\rangle_{L^{2}(X)}+\|\Theta(E)\|_{L^{n}(X)}(1+\|\theta\|^{2}_{L^{\infty}(X)})\|\na_{E}\a\|^{2}_{L^{2}(X)}\\
\end{split}
\end{equation}
The last line, we use the inequality (See Theorem \ref{T4})
\begin{equation}\nonumber
\begin{split}
\|\a\|_{L^{2}(X)}&\lesssim\|\theta\|_{L^{\infty}(X)}\|\na|\a|\|_{L^{2}(X)}\lesssim\|\theta\|_{L^{\infty}(X)}\|\na_{E}\a\|_{L^{2}(X)}.\\
\end{split}
\end{equation}
If the curvature $\Theta(E)$ obeys $\|\Theta(E)\|_{L^{n}(X)}\leq\varepsilon$, where $\varepsilon$ is the constant in the Lemma \ref{L4}, then we have
\begin{equation}\nonumber
\begin{split}
\|\na_{E}\a\|^{2}_{L^{2}(X)}&\lesssim(\langle\De_{E}\a,\a\rangle_{L^{2}(X)}+\|\a\|^{2}_{L^{2}(X)})\\
&\lesssim (1+\|\theta\|^{2}_{L^{\infty}(X)})\langle\De_{E}\a,\a\rangle_{L^{2}(X)}\\
\end{split}
\end{equation}
here we use the inequality (See Theorem \ref{T3})
$$\|\a\|^{2}_{L^{2}(X)}\lesssim\|\theta\|^{2}_{L^{\infty}(X)}\langle\De_{E}\a,\a\rangle_{L^{2}(X)}.$$
This yields the desired estimate
\begin{equation}\label{E13}
\langle\De_{E}\a,\a\rangle_{L^{2}(X)}\lesssim\langle\De_{\bar{\pa}_{E}}\a,\a\rangle_{L^{2}(X)}
+\|\Theta(E)\|_{L^{n}(X)}(1+\|\theta\|^{2}_{L^{\infty}(X)})^{2}\langle\De_{E}\a,\a\rangle_{L^{2}(X)}.
\end{equation}
If we replace the $L^{n}$-norm of $\theta(E)$ to $L^{\infty}$-norm, then inequalities (\ref{E12}) and (\ref{E13}) replaced to
	\begin{equation}\nonumber
	\begin{split}
	\langle\De_{E}\a,\a\rangle_{L^{2}(X)}&\lesssim\langle\De_{\bar{\pa}_{E}}\a,\a\rangle_{L^{2}(X)}+\|\Theta(E)\|_{L^{\infty}(X)}\|\a\|^{2}_{L^{2}(X)}\\
	&\lesssim\langle\De_{\bar{\pa}_{E}}\a,\a\rangle_{L^{2}(X)}+\|\Theta(E)\|_{L^{\infty}(X)}\|\theta\|^{2}_{L^{\infty}(X)}\|\na_{E}\a\|^{2}_{L^{2}(X)}\\
	\end{split}
	\end{equation}
	and	
	\begin{equation}\nonumber
	\langle\De_{E}\a,\a\rangle_{L^{2}(X)}\lesssim\langle\De_{\bar{\pa}_{E}}\a,\a\rangle_{L^{2}(X)}
	+\|\Theta(E)\|_{L^{\infty}(X)}\|\theta\|^{2}_{L^{\infty}(X)}(1+\|\theta\|^{2}_{L^{\infty}(X)})\langle\De_{E}\a,\a\rangle_{L^{2}(X)}.
	\end{equation}
Hence if $\|\Theta(E)\|_{L^{n}(X)}$ or $\|\Theta(E)\|_{L^{\infty}(X)}$ is sufficiently small, we have
$$\langle\De_{E}\a,\a\rangle_{L^{2}(X)}\lesssim\langle\De_{\bar{\pa}_{E}}\a,\a\rangle_{L^{2}(X)}.$$
At last, if $\a\in\mathcal{H}^{p,q}_{(2)}(X,E)$, then $\a$ is also in $\mathcal{H}^{k}_{(2)}(X,E)$, $k=p+q$. Therefore, following the Theorem \ref{T3}, we complete the proof of this theorem. 
\end{proof}

\subsection{Nonvanishing results}
 In \cite{Gro}, Gromov proved a nonvanishing for $p+q=dim_{\C}X$ follows from the $L^{2}$-index theorem and an upper bound for the bottom of the spectrum \cite[Main Theorem]{Gro}.  A special case of a conjecture of Hopf follows from the main theorem. Namely, the Euler characteristic $\chi(X)$ of a compact, negatively curved K\"{a}hler manifold $X$ of complex dimension $n$ satisfies $sign\chi(X)=(-1)^{n}$. Let $E$ be a vector bundle equipped with a Hermitian metric and Hermitian connection $d_{A}$ over a compact K\"{a}hler manifold. If $P:=d_{A}+d_{A}^{\ast}$, the Atiyah-Singer's index theorem states 
$$Index(P)=\int_{X}\mathcal{L}(X)ch(E).$$
Here $\mathcal{L}(X)$ is Hirzebruch's $\mathcal{L}$ classes,
$$\mathcal{L}(X)=1+\ldots+e(X),$$
where $1\in H^{0}(X)$ and $e(X)\in H^{2n}(X)$ is the Euler class. The, the class $\mathcal{L}(X)ch(E)$ has in general non-trivial components in various degrees. What is meant by the integral on the right hand side, of course, is the evaluation of the top degree component
$[\mathcal{L}(X)ch(E)]_{2n}=[\mathcal{L}_{k}(X)ch_{2n-k}(E)]$. 

For any $p>0$, we have an estimate
$$\|F_{A}\|_{L^{p}(X)}\leq \|F_{A}\|_{L^{\infty}(X)}Vol(X)^{1/p}.$$
Following a well-known result due to Uhlenbeck for the connections  with $L^{p}$-small curvature ($2p>\dim_{\mathbb{R}}X$) \cite{Uhl}, we can provide the $L^{\infty}$-norm of the curvature $F_{A}$ of the Hermitian connection $d_{A}$ to ensure  that
$$ \|F_{A}\|_{L^{\infty}(X)}\leq\varepsilon,$$
where $\varepsilon$ is a small enough positive constant on \cite[Corollary 4.3]{Uhl}, then there exists a flat connection $d_{\Gamma}$ on the Hermitian vector bundle $E$. We can write the connection $d_{A}=d_{\Gamma}+a$, where $a$ is a $1$-form take value in $\mathfrak{g}_{E}$. Therefore, $F_{A}=d_{\Gamma}a+a\wedge a$. Then, $[ch(E)]=[Tr(\exp \frac{\textit{i}}{2\pi}F_{A})]=rank(E)$, i.e, there exists a differential form $\eta$ such $$ch(E)=rank(E)+d\eta$$
Noting that $d\mathcal{L}(X)=0$. We then have, 
\begin{equation}\nonumber
\begin{split}
Index(P)&=\int_{X}(1+\ldots+e(X))(rank(E)+d\eta)\\
&=rank(E)\int_{X}e(X)+\int_{X}d(\mathcal{L}(X)\eta)\\
&=rank(E)\chi(X).\\
\end{split}
\end{equation}
Therefore, if $X$ is a closed K\"{a}hler hyperbolic manifold, then
$$Index(P)\neq0.$$
We denote  by $\tilde{P}$ the lift of $P$ to the universal convering space $\tilde{X}$ of the closed hyperbolic K\"{a}hler $X$.  Let $\Ga=\pi_{1}(M)$. Then the Atiyah's $L^{2}$ index \cite{Ati,Pan} gives that  the $L^{2}$ kernel of $\tilde{P}$ has a finite $\Ga$-dimension and 
$$L^{2}Index_{\Ga}(\tilde{P})=Index (P)\neq 0.$$
This implies that  $\tilde{P}$ has a non trivial $L^{2}$ kernel. Following the  vanishing theorem \ref{T3}, we have $$\mathcal{H}^{n}_{(2)}(\tilde{X},\tilde{E}):=\{\a\in\Om^{n}_{(2)}(\tilde{X},\tilde{E}):\De_{\tilde{A}}\a=0\} \neq\{0\}.$$
where $\tilde{E}$, $d_{\tilde{A}}$ are the lifted bundle and connection of $E$, $d_{A}$.
\begin{corollary}\label{C2}
Let $(X,\w)$ be a closed K\"{a}hler hyperbolic manifold, $\dim_{\C}X=n$, $E$ be a Hermitian vector bundle over $X$ and $\pi:(\tilde{X},\tilde{\w})\rightarrow(X,\w)$ be the universal covering. Then there exists a positive constant $\varepsilon$ with following significance. If the curvature $F_{A}$ of  a Hermitian connection $d_{A}$ on $E$ obeys
$$\|F_{A}\|_{L^{\infty}(X)}\leq \varepsilon,$$
then $$\mathcal{H}^{n}_{(2)}(\tilde{X},\tilde{E}) \neq\{0\}.$$
\end{corollary}
\begin{proof}
Since $\pi$ is a local isometry, $|F_{\tilde{A}}|_{\tilde{g}}=|\pi^{\ast}(F_{A})|_{\tilde{g}}=|F_{A}|_{g}$ and the universal covering space $\tilde{X}$ is bounded geometry. Thus, the Riemannian curvature tensor of the metric $\tilde{g}$ must be bounded. By the hypothesis of $F_{A}$, we have
$$\|F_{\tilde{A}}\|_{L^{\infty}(\tilde{X})}\leq\varepsilon.$$
Then the conclusion follows from the above argument. 
\end{proof}
Suppose $X$ is a compact K\"{a}hler manifold with underlying Riemann metric $g$ . We denote by $\na^{g}$ the Hermitian connections induced by the Levi-Civita connection on $\Om^{\bullet,\bullet}TX$. Let $E$  be a holomorphic vector bundle $E$ on a compact K\"{a}hler manifold $X$, $D_{E}$ be the Chern-connection on $E$. Thus the twist bundle $\Om^{p,0}TX\otimes E$ is also a holomorphic vector bundle on $X$. Let $$\chi_{p}(X,E):=\sum_{q\geq0}(-1)^{i}h^{p,q}(E)$$ denote the index of the operator $$\mathcal{D}_{p}=\bar{\pa}_{E}+\bar{\pa}_{E}^{\ast}:\Om^{p,\ast}(X,E)\rightarrow\Om^{p,\ast\pm1}(X,E),$$ where  
$$h^{p,q}(E)=\dim\mathcal{H}^{p,q}(X,E)=\dim\{\a\in\Om^{p,q}(X,E):\mathcal{D}_{p}\a=0  \}.$$ 
In particular, $\chi_{0}(X,E)$ called the \textit{Euler-Poincar\'{e}} characteristic \cite[Section 5]{Huy}. The Hirzebruch-Riemann-Roch theorem gives 
$$\chi_{p}(X,E)=\int_{X}td(X)ch(\Om^{p,0}TX\otimes E)=\int_{X}td(X) ch(\Om^{p,0}TX)ch(E).$$
If $X$ is also hyperbolic, then Gromov proved that
\begin{theorem}\cite[0.4.A. Theorem]{Gro}\label{T2}
If $X$ is K\"{a}hler hyperbolic, $\dim_{\C}X=n$, then for every $p=0,1,\ldots,n$, the Euler characteristic $$\chi_{p}(X)=\int_{X}td(X)ch(\Om^{p,0}TX)$$
does not vanish and $sign\chi_{p}=(-1)^{n-p}$.
\end{theorem}
Following the idea in Corollary \ref{C2}, we have 
\begin{corollary}
Let $(X,\w)$ be a closed K\"{a}hler hyperbolic manifold, $\dim_{\C}X=n$, $E$ be a Hermitian vector bundle over $X$ and $\pi:(\tilde{X},\tilde{\w})\rightarrow(X,\w)$ be the universal covering. Then there exists a positive constant $\varepsilon$ with following significance. If the curvature $F_{A}$ of the Chern connection $D_{E}$ on $E$ obeys
$$\|\Theta(E)\|_{L^{\infty}(X)}\leq \varepsilon,$$
then  for any $0\leq p\leq n$,
$$\mathcal{H}^{p,n-p}_{(2)}(\tilde{X},\tilde{E}) \neq\{0\},$$
\end{corollary}
\begin{proof}
Since $\pi$ is a local isometry, the universal covering space $\tilde{X}$ is bounded geometry.  The Chern curvature $\Theta(\tilde{E})$ also satisfies
$$\|\Theta(\tilde{E})\|_{L^{\infty}(\tilde{X})}\leq\varepsilon.$$
We denote by $\tilde{\mathcal{D}}_{p}$ the lifted of  $\mathcal{D}_{p}$ for $p\geq0$. Noting that $d(td(X)ch(\Om^{p,0}TX))=0$ and $[ch(E)]=rank(E)$. We then have, 
\begin{equation}\nonumber
Index(\mathcal{D}_{p})=\int_{X}td(X)ch(\Om^{p,0}TX)(rank(E)+d\eta)=rank(E)\chi_{p}(X).
\end{equation}
Then, by the Atyah's $L^{2}$-index theorem and Gromov's theorem \ref{T2}, we have $$L^{2}Index_{\Gamma}(\tilde{\mathcal{D}}_{p})\neq 0.$$
This implies that  $\tilde{\mathcal{D}}_{p}$ has a non trivial $L^{2}$ kernel. Following the  vanishing theorem \ref{T3}, we have $\mathcal{H}^{p,n-p}_{(2)}(\tilde{X},\tilde{E})\neq \{0\}$.
\end{proof}
\section{The case of a K\"{a}hler surface}
Following the method on Section 2.2, any $\a\in\Om^{2}(X,E)$ on a K\"{a}hler surface can be decomposed as $\a=\a^{0,2}+\a^{2,0}+\a^{0}\otimes\w+\a^{1,1}_{0},$ where $\a^{0}=\frac{1}{2}\La\a$,\ $\La\a^{1,1}_{0}=0$. Furthermore, over a $4$-dimensional Riemannnian manifold, $\Om^{2}(X,E)$ is decomposed into its self-dual and anti-self-dual components, $\Om^{2}(X,E)=\Om^{+}(X,E)\oplus\Om^{-}(X,E)$,
where $\Om^{\pm}$ denotes the projection onto the $\pm1$ eigenspace of the Hodge star operator $\ast:\Om^{2}\rightarrow\Om^{2}$. Hence $F_{A}$ can be decomposed into its self-dual and anti-self-dual components $ F_{A}=F^{+}_{A}+F^{-}_{A}$. A connection is called self-dual (respectively anti-self-dual), if $F_{A}=F^{+}_{A}$ (respectively $F_{A}=F^{-}_{A}$). The self-dual part $F^{+}_{A}$ is given as $F^{+}_{A}=F^{2,0}_{A}+F^{0,2}_{A}+\frac{1}{2}(\La F_{A})\otimes\w$ and the anti-self-dual part $F^{-}_{A}$ is a form of type $(1,1)$ which is orthogonal to $\w$. A connection is called an instanton if it is either self-dual or anti-self-dual. On compact oriented $4$-manifolds, an instanton is always an absolute minimizer of the Yang-Mills energy. Not all Yang-Mills connections are instantons (See \cite{SaSe,SSU}).

We will construct another priori estimate  under the $L^{2}$-norm of the self-dual part of the curvature is small enough.
\begin{lemma}\label{L5}
Let $(X,\w)$ be a complete K\"{a}hler surface with a K\"{a}hler $d$(bounded) form $\w$, i.e., there exists a bounded $1$-form $\theta$ such that $\w=d\theta$, $E$ be a Hermitian vector bundle over $X$ and $d_{A}$ be a Hermitian connection on $E$. Then there are constants $\varepsilon=\varepsilon(X,\theta)\in(0,1]$ and $C=C(X,\theta)\in[1,\infty)$ with following significance. If the curvature $F_{A}$ of the connection $d_{A}$ obeys
$$\|F^{+}_{A}\|_{L^{2}(X)}\leq \varepsilon$$
Then any $\a\in\mathcal{H}^{2}_{(2)}(X,E)$ satisfies
$$\langle\De_{A}\a^{0,2},\a^{0,2}\rangle_{L^{2}(X)}\leq C\|F^{0,2}_{A}\|^{2}_{L^{2}(X)}(1+\|\theta\|^{2}_{L^{\infty}(X)})^{2}\|\na_{A}\a^{0,2}\|^{2}_{L^{2}(X)}.$$
\end{lemma}
\begin{proof}
Following identity (i) in Proposition \ref{P1}, we have
$$\bar{\pa}_{A}^{\ast}\bar{\pa}_{A}^{\ast}\a^{0,2}=\sqrt{-1}\bar{\pa}^{\ast}_{A}\bar{\pa}_{A}\a_{0}.$$
The Weizenb\"{o}ck formula gives, see \cite[Lemma 6.1.7]{Donaldson/Kronheimer},
\begin{equation}\nonumber
\begin{split}
\De_{A}\a_{0}&=2\De_{\bar{\pa}_{A}}\a_{0}+\sqrt{-1}\La F_{A}(\a_{0})\\
&=-2\sqrt{-1}\bar{\pa}_{A}^{\ast}\bar{\pa}_{A}^{\ast}\a^{0,2}+\sqrt{-1}\La F_{A}(\a_{0}).\\
\end{split}
\end{equation}
Next, we observe that
\begin{equation}\nonumber
\begin{split}
|\langle\bar{\pa}_{A}^{\ast}\bar{\pa}_{A}^{\ast}\a^{0,2},\a_{0}\rangle_{L^{2}(X)}&=|\langle\ast(F_{A}^{0,2}\wedge\ast\a^{0,2}),\a_{0}\rangle_{L^{2}(X)}|\\
&\leq\|F_{A}^{0,2}\|_{L^{2}(X)}\|\a^{0,2}\|_{L^{4}(X)}\|\a_{0}\|_{L^{4}(X)}\\
&\leq\varepsilon\|\a_{0}\|^{2}_{L^{4}(X)}+\frac{1}{2\varepsilon}\|F_{A}^{0,2}\|^{2}_{L^{2}(X)}\|\a^{0,2}\|^{2}_{L^{4}(X)},
\end{split}
\end{equation}
where $\varepsilon$ is a positive constant. We also observe that
$$|\langle\La F_{A}(\a_{0}),\a_{0}\rangle_{L^{2}(X)}|\leq\|\La F_{A}\|_{L^{2}(X)}\|\a_{0}\|^{2}_{L^{4}(X)}.$$
This yields the desired estimate
\begin{equation}\nonumber
\begin{split}
\langle\De_{A}\a_{0},\a_{0}\rangle_{L^{2}(X)}&\leq\frac{1}{2\varepsilon}\|F_{A}^{0,2}\|^{2}_{L^{2}(X)}\|\a^{0,2}\|^{2}_{L^{4}(X)}+(\varepsilon+\|\La F_{A}\|_{L^{2}(X)})\|\a_{0}\|^{2}_{L^{4}(X)}\\
&\lesssim(\varepsilon+\|\La F_{A}\|_{L^{2}(X)})\|\a_{0}\|^{2}_{L^{2}_{1}(X)}+\frac{1}{2\varepsilon}\|F_{A}^{0,2}\|^{2}_{L^{2}(X)}\|\a^{0,2}\|^{2}_{L^{4}(X)} \\
&\lesssim(\varepsilon+\|\La F_{A}\|_{L^{2}(X)})(1+\|\theta\|^{2}_{L^{\infty}(X)})\|\na_{A}\a_{0}\|^{2}_{L^{2}(X)}\\
&+\frac{1}{2\varepsilon}\|F_{A}^{0,2}\|^{2}_{L^{2}(X)}\|\a^{0,2}\|^{2}_{L^{4}(X)},
\end{split}
\end{equation}
where we use the fact, See the second item of Theorem \ref{T3},
\begin{equation}
\|\a_{0}\|^{2}_{L^{2}(X)}\lesssim\|\theta\|^{2}_{L^{\infty}(X)}\|\na_{A}\a_{0}\|^{2}_{L^{2}(X)}=\|\theta\|^{2}_{L^{\infty}(X)}\langle\na^{\ast}_{A}\na_{A}\a_{0},\a_{0}\rangle_{L^{2}(X)}
\end{equation}
and  Sobolev embedding $L^{2}_{1}\hookrightarrow L^{4}$. We can provide $(\varepsilon+\|\La F_{A}\|_{L^{2}(X)})(1+\|\theta\|^{2}_{L^{\infty}(X)})$ sufficiently small to ensure that
\begin{equation}\label{E9}
\langle\De_{A}\a^{0},\a^{0}\rangle_{L^{2}(X)}\lesssim\|F_{A}^{0,2}\|^{2}_{L^{2}(X)}\|\a^{0,2}\|^{2}_{L^{4}(X)}.
\end{equation}
Following the identity (ii) in Proposition \ref{P1}, we also have
$$\bar{\pa}_{A}\bar{\pa}_{A}^{\ast}\a^{0,2}=\sqrt{-1}\bar{\pa}_{A}\bar{\pa}_{A}\a_{0}.$$
On a direct calculation, we have
$$\De_{A}=2\De_{\bar{\pa}_{A}}+\sqrt{-1}[\La,F^{1,1}_{A}]+\sqrt{-1}[\La,F^{0,2}_{A}]-\sqrt{-1}[\La,F^{2,0}_{A}].$$
Noting that for any $\a^{0,2}\in\Om^{0,2}(X,E)$,  $$\sqrt{-1}[\La,F^{1,1}_{A}]\a^{0,2}=\sqrt{-1}[\La,F^{0,2}_{A}]\a^{0,2}=0$$
and 
$$\langle\La(F^{2,0}_{A}\wedge\a^{0,2}),\a^{0,2}\rangle_{L^{2}(X)}=\langle F^{2,0}_{A}(\La\a^{0,2}),\a^{0,2}\rangle_{L^{2}(X)}=0.$$  
Hence we have
\begin{equation}\label{E4.4}
\langle\De_{A}\a^{0,2},\a^{0,2}\rangle_{L^{2}(X)}=2\langle\De_{\bar{\pa}_{A}}\a^{0,2},\a^{0,2}\rangle_{L^{2}(X)}=2\langle\sqrt{-1} [F^{0,2}_{A}(\a_{0}),\a^{0,2}\rangle_{L^{2}(X)}.
\end{equation} 
Next, we also observe that
\begin{equation}\nonumber
\begin{split}
|\langle F_{A}^{0,2}(\a_{0}),\a^{0,2}\rangle_{L^{2}(X)}|&\leq\|F_{A}^{0,2}\|_{L^{2}(X)}\|\a^{0,2}\|_{L^{4}(X)}\|\a_{0}\|_{L^{4}(X)}\\
&\leq\|\a_{0}\|^{2}_{L^{4}(X)}+\frac{1}{2}\|F_{A}^{0,2}\|^{2}_{L^{2}(X)}\|\a^{0,2}\|^{2}_{L^{4}(X)}\\
&\lesssim(\|\na_{A}\a_{0}\|^{2}_{L^{2}(X)}+\|\a_{0}\|^{2}_{L^{2}(X)})+\frac{1}{2}\|F_{A}^{0,2}\|^{2}_{L^{2}(X)}\|\a^{0,2}\|^{2}_{L^{4}(X)}\\
&\lesssim(1+\|\theta\|^{2}_{L^{\infty}(X)})\|\na_{A}\a_{0}\|^{2}_{L^{2}(X)}+\frac{1}{2}\|F_{A}^{0,2}\|^{2}_{L^{2}(X)}\|\a^{0,2}\|^{2}_{L^{4}(X)}.\\
\end{split}
\end{equation}
This yields the desired estimation
\begin{equation}\nonumber
\begin{split}
\langle\De_{A}\a^{0,2},\a^{0,2}\rangle_{L^{2}(X)}&\lesssim\|F^{0,2}_{A}\|^{2}_{L^{2}(X)}\|\a^{0,2}\|^{2}_{L^{4}(X)}+(1+\|\theta\|^{2}_{L^{\infty}(X)})\|\na_{A}\a_{0}\|^{2}_{L^{2}(X)}\\
&\lesssim\|F_{A}^{0,2}\|^{2}_{L^{2}(X)}\|\a^{0,2}\|^{2}_{L^{4}(X)}
+(1+\|\theta\|^{2}_{L^{\infty}(X)})\|F_{A}^{0,2}\|^{2}_{L^{2}(X)}\|\a^{0,2}\|^{2}_{L^{4}(X)}\\
&\lesssim\|F_{A}^{0,2}\|^{2}_{L^{2}(X)}(1+\|\theta\|^{2}_{L^{\infty}(X)})\|\a^{0,2}\|^{2}_{L^{2}_{1}(X)}.
\end{split}
\end{equation}
Following Theorem \ref{T4} and Kato inequality $|\na|\a^{0,2}||\leq |\na_{A}\a^{0,2}|$, we have
\begin{equation}\label{E4.2}
\begin{split}
\|\a^{0,2}\|_{L^{2}(X)}&\lesssim \|\theta\|_{L^{\infty}(X)}\|\na|\a^{0,2}|\|_{L^{2}(X)}\lesssim \|\theta\|_{L^{\infty}(X)}\|\na_{A}\a^{0,2}\|_{L^{2}(X)}.\\
\end{split}
\end{equation}
By combining these estimates for $\a^{0,2}$, we complete this proof.
\end{proof}
\begin{lemma}\label{L6}
Assume the hypothesis of Lemma \ref{L5}. We denote by $S^{\pm}$ the positive (resp. negative) part of  the scalar curvature $S$ of metric $g$. Then there are constants $\varepsilon=\varepsilon(X,\theta)\in(0,1]$ and $C=C(X,\theta)\in[1,\infty)$ with following significance.  If the curvature $F_{A}$ of the connection $d_{A}$ and $S^{-}$ satisfies
$$\|\La F_{A}\|_{L^{2}(X)}+\|S^{-}\|_{L^{2}(X)}\leq \varepsilon$$
then for any $\a\in\Om^{2}_{(2)}(X,E)$, we have an inequality 
\begin{equation}\label{E4.3}
\|\na_{A}\a^{0,2}\|^{2}_{L^{2}(X)}\leq C\langle\De_{A}\a^{0,2},\a^{0,2}\rangle_{L^{2}(X)}.
\end{equation}
\end{lemma}
\begin{proof}
For any $\a^{0,2}\in\Om^{0,2}(X,E)$, the Weizenb\"{o}ck formula gives, See \cite{Ito},
$$\De_{\bar{\pa}_{A}}\a^{0,2}=\na^{\ast}_{A}\na_{A}\a^{0,2}+S\a^{0,2}+[\sqrt{-1}\La F_{A},\a^{0,2}],$$
where $S$ is the scalar curvature of $g$. Next, we have
\begin{equation}\nonumber
\begin{split}
\|\na_{A}\a^{0,2}\|^{2}_{L^{2}(X)}&\leq\langle\De_{\bar{\pa}_{A}}\a^{0,2},\a^{0,2}\rangle_{L^{2}(X)}+|\langle[\sqrt{-1}\La F_{A},\a^{0,2}],\a^{0,2}\rangle_{L^{2}(X)}|-\int_{X}(S^{+}+S^{-})|\a^{0,2}|^{2}\\
&\lesssim\langle\De_{\bar{\pa}_{A}}\a^{0,2},\a^{0,2}\rangle_{L^{2}(X)}+\|\La F_{A}\|_{L^{2}(X)}\|\a^{0,2}\|^{2}_{L^{4}(X)}-\int_{X}S^{-}|\a^{0,2}|^{2}\\
&\lesssim\langle\De_{\bar{\pa}_{A}}\a^{0,2},\a^{0,2}\rangle_{L^{2}(X)}+(\|\La F_{A}\|_{L^{2}(X)}+\|S^{-}\|_{L^{2}(X)})\|\a^{0,2}\|^{2}_{L^{2}_{1}(X)}\\
&\lesssim\langle\De_{\bar{\pa}_{A}}\a^{0,2},\a^{0,2}\rangle_{L^{2}(X)}+(\|\La F_{A}\|_{L^{2}(X)}+\|S^{-}\|_{L^{2}(X)})(1+\|\theta\|^{2}_{L^{\infty}(X)})\|\na_{A}\a^{0,2}\|^{2}_{L^{2}(X)}\\
&\lesssim\langle\De_{A}\a^{0,2},\a^{0,2}\rangle_{L^{2}(X)}+(\|\La F_{A}\|_{L^{2}(X)}+\|S^{-}\|_{L^{2}(X)})(1+\|\theta\|^{2}_{L^{\infty}(X)})\|\na_{A}\a^{0,2}\|^{2}_{L^{2}(X)}\\
\end{split}
\end{equation}
here we use the inequality (\ref{E4.2}) and identity (\ref{E4.4}). Hence we can provide $(\|\La F_{A}\|_{L^{2}(X)}+\|S^{-}\|_{L^{2}(X)})$ sufficiently small to ensure inequality (\ref{E4.3}). 
\end{proof}
\begin{proof}[\textbf{Proof of Theorem \ref{T1.6}}]
For any $\a\in\mathcal{H}^{2}_{(2)}(X,E)$, we can decompose the $2$-form $\a$ as $\a=\a^{0,2}+\a^{2,0}+\a_{0}\otimes\w+\a_{0}^{1,1}$. If the scalar curvature $S$ is nonnegative and the curvature $F_{A}$ obeys 
$$\|F^{+}_{A}\|_{L^{2}(X)}\leq\varepsilon,$$
where $\varepsilon$ is a positive constant less than the constants on Lemma \ref{L5} and \ref{L6}. Therefore, combining the inequalities on Lemma \ref{L5} and \ref{L6}, we have the desired estimation
\begin{equation}\nonumber
\begin{split}
\langle\De_{A}\a^{0,2},\a^{0,2}\rangle_{L^{2}(X)}&\lesssim \|F^{0,2}_{A}\|^{2}_{L^{2}(X)}(1+\|\theta\|^{2}_{L^{\infty}(X)})^{2}\|\na_{A}\a^{0,2}\|^{2}_{L^{2}(X)}\\
&\lesssim \|F_{A}^{0,2}\|^{2}_{L^{2}(X)}(1+\|\theta\|^{2}_{L^{\infty}(X)})^{2}\langle\De_{A}\a^{0,2},\a^{0,2}\rangle_{L^{2}(X)}.\\
\end{split}
\end{equation}
Since $\|F_{A}^{0,2}\|^{2}_{L^{2}(X)}\leq\|F_{A}^{+}\|_{L^{2}(X)}$, we can provide $\varepsilon$ sufficiently small to ensure $\a^{0,2}$ is $L^{2}$-harmonic form. Then, by the Equation (\ref{E4.3}), we get  $\na_{A}\a^{0,2}\equiv0$. The Kato inequality implies that $\na|\a^{0,2}|=0$. Therefore, the Gromov's vanishing theorem \ref{T4} implies that $|\a^{0,2}|=0$, i.e, $\a^{0,2}=0$. At last, by the Equation (\ref{E9}), $\a_{0}$ also vanishes. Therefore, any $L^{2}$-harmonic form on $E$ is anti-self-dual.
\end{proof} 
\begin{proof}[\textbf{Proof of Corollary \ref{C1}}]
Suppose now that the Hermitian vector $E$ of rank $r$, is equipped with a Hermitian  connection $d_{A}$. There are unique Hermitian connections  $d_{A^{\ast}}$ on $E^{\ast}$ and $d_{EndE}:=d_{A\otimes A^{\ast}}$ on $EndE=E\otimes E^{\ast}$ induced by connection $d_{A}$. We then have  $F_{A\otimes A^{\ast}}=F_{A}\otimes Id_{E^{\ast}}+Id_{E}\otimes F_{A^{\ast}}$ and $F_{A^{\ast}}=-F^{\dagger}_{A}$, where $\dagger$ is the transposition operator $EndE\rightarrow End E^{\ast}$ \cite[Charp V, Section 4]{Dem}. Therefore,
$$| F_{A\otimes A^{\ast}}|\leq|F_{A}\otimes Id_{E^{\ast}}|+|Id_{E}\otimes F_{A^{\ast}}|=2\sqrt{r}|F_{A}|.$$
We can choose the $L^{n}$-norm of the curvature $F_{A}$ small enough to ensure $$\|F_{A\otimes A^{\ast}}\|_{L^{n}(X)}\leq\varepsilon,$$ 
where $\varepsilon$ is the constant on Theorem \ref{T3} and \ref{T1.6}. Therefore, $\mathcal{H}_{(2)}^{2}(X,EndE)=0$, for $n\geq3$ and  the $L^{2}$-harmonic form on $End E$ is anti-self-dual, for $n=2$. Suppose also that the connection $A$ is a Yang-Mills connection with $L^{2}$-norm curvature. Noting that $F_{A}$ is a harmonic $2$-form on the bundle $\mathfrak{g}_{E}\subset End E$. Hence $F_{A}=0$ for $n\geq 3$ and $F^{+}_{A}=0$ for $n=2$. We complete the proof of this result.  
\end{proof}
\subsection*{Acknowledgements}
We would like to thank Gromov for kind comments regarding his article \cite{Gro}.  We would also like to thank the anonymous referee for  careful reading of my manuscript and helpful comments. In particular, the referees pointed out that we can study the nonvanishing theorem in the small enough $L^{\infty}$-norm case. This work is supported by Nature Science Foundation of China No. 11801539.

\end{document}